\documentclass[11pt, reqno]{amsart}

\usepackage{amsmath,amsfonts,amssymb,amsthm, bbold}
\usepackage{graphics,esint,bbm}
\usepackage{epsfig}
\usepackage{xcolor}

\usepackage{xfrac}
\usepackage{comment}
\usepackage{hyperref}

\newtheorem{definition}{Definition}[section]
\newtheorem{proposition}{Proposition}[section]
\newtheorem{theorem}{Theorem}[section]
\newtheorem{remark}{Remark}[section]
\newtheorem{lemma}{Lemma}[section]
\newtheorem{corollary}{Corollary}[section]

\newcommand{\R}{\mathbb{R}}
\newcommand{\Z}{\mathbb{Z}}

\newcommand{\supp}{\mathrm{supp }}

\def\T{\mathbb{T}^3}
\def\N{\mathbb N}
\def\eps{\varepsilon}
\def\1{{\mathchoice {1\mskip-4mu\mathrm l}      
		{1\mskip-4mu\mathrm l} 
		{1\mskip-4.5mu\mathrm l} {1\mskip-5mu\mathrm l}}}

\def\mR{\mathring{R}}

\def\Id{\mathrm{Id}}
 
\def\div{\mathrm{div\,}}
\def\curl{\mathrm{curl\,}}
\def\tr{\mathrm{tr\, }}

\numberwithin{equation}{section}

\begin{document}

\title[Non-uniqueness for the Euler equations]{Non-uniqueness for the Euler equations up to Onsager's critical exponent}

\author{Sara Daneri}
\address{Gran Sasso Science Institute, 67100 L'Aquila, Italy}
\email{sara.daneri@gssi.it}

\author{Eris Runa}
\address{Quant Institute, Deutsche Bank AG, 10585 Berlin, Germany}
\email{eris.runa@gmail.com}

\author{L\'aszl\'o Sz\'ekelyhidi Jr.} 
\address{Institut f\"ur Mathematik, Universit\"at Leipzig, 04009 Leipzig, Germany}
\email{szekelyhidi@math.uni-leipzig.de}

\maketitle
\begin{abstract}
	In this paper we deal with the Cauchy problem for the incompressible Euler equations in the three-dimensional periodic setting. We prove non-uniqueness for an $L^2$-dense set of H\"older continuous initial data in the class of H\"older continuous admissible weak solutions for all exponents below the Onsager-critical $1/3$. This improves previous results on non-uniqueness obtained in \cite{Dan,DanSz} and generalizes \cite{BDSV}. 
	\end{abstract}

\section{Introduction}

In this paper we address the Cauchy problem for the incompressible Euler equations
\begin{equation}
\label{eq:EE}
\left\{\begin{aligned}
\partial_tv+\div(v\otimes v)+\nabla p&=0 && &\text{in $\T\times(0,T)$}\\
\div v&=0 && &\text{in $\T\times(0,T)$}\\
v(\cdot,0)&=v_0(\cdot) && &\text{in $\T$}
\end{aligned}
\right.
\end{equation}
on the three-dimensional torus $\T$, where $v:\T\times[0,T]\to\R^3$ is the velocity field of the fluid and $p:\T\times[0,T]\to\R$ the pressure field.

We are interested in \emph{admissible} weak solutions to \eqref{eq:EE}, namely weak solutions $v\in C([0,T];L^2_w(\T))$ such that
\begin{equation}\label{eq:admineq}
\int_{\T}|v(x,t)|^2\,dx\leq\int_{\T}|v_0|^2\,dx.
\end{equation} 
The above is a very natural physical condition, which assuming the velocity field is in $C^1$ (namely the solution is classical) implies uniqueness among all weak solutions which satisfy \eqref{eq:admineq}. This is the well-known weak-strong uniqueness phenomenon, which holds even among measure-valued solutions \cite{BDS}.
For $L^\infty$ weak solutions, it has been instead shown in \cite{DS10} that  infinitely many admissible solutions  can have the same initial datum. Such $L^\infty$ initial data are the so-called ``wild'' initial data and are dense in $L^2$ (see \cite{SzW}).

A natural question is whether there exists a regularity threshold above which admissibility implies uniqueness and below which non-uniqueness may occur. We treat this question in the class of $C^\beta$-weak solutions, that is, weak solutions which are H\"older continuous in space with exponent $\beta$, so that 
\begin{equation}\label{eq:holder}
|v(x,t)-v(y,t)|\leq C|x-y|^\beta \qquad\forall\,t\in[0,T],\,x,y\in\T
\end{equation}
for some constant $C$. According to the celebrated Onsager's conjecture \cite{Ons}, $C^\beta$-weak solutions of the Euler equations conserve the total kinetic energy if $\beta>1/3$, but anomalous dissipation of energy may be present if $\beta<1/3$. Recently this conjecture has been fully resolved (we refer to \cite{Eyi,CET} for the case $\beta>1/3$ and to \cite{Ise,BDSV} for the case $\beta<1/3$, and the extensive references therein). Our aim is to extend the results in \cite{Ise,BDSV} and show that ``wild'' initial data is $L^2$-dense in the class of $C^\beta$-weak solutions, which are admissible in the sense of \eqref{eq:admineq}. To state our result more precisely, we introduce the following
\begin{definition}\label{def:wild}
	Given a divergence-free vector field $v_0\in C^{\beta_0}(\T)$, we say that $v_0$ is a \emph{wild initial datum in $C^\beta$} if there exist infinitely many weak solutions $v$ to \eqref{eq:EE} on $\T\times[0,T]$ and satisfying \eqref{eq:admineq} and \eqref{eq:holder}.
\end{definition}

Our main result is the following.
\begin{theorem}\label{thm:main}
	For any $0<\beta<1/3$, the set of divergence-free vector fields $v_0\in C^\beta(\T;\R^3)$ which are wild initial data in $C^\beta$ is a dense subset of the divergence-free vector fields in $L^2(\T;\R^3)$.
\end{theorem}

Previous work on existence and density of wild initial data has been done in \cite{DS10,SzW} for bounded $L^\infty$ weak solutions, and in \cite{Dan,DanSz} for H\"older continuous weak solutions. The underlying idea is the following: iteration schemes based on convex integration, as in \cite{DS09,DS13,DS14,BDIS,Ise,BDSV}, start with a subsolution (see Section \ref{sect:subs} below) and, by a sequence of high-frequency perturbations   produce weak solutions of the Euler equations in the limit. Thus, analogously to the celebrated Nash-Kuiper isometric embedding theorem \cite{Nash} (see also \cite{ContiDSz}), such schemes not only produce one weak solution, but automatically a whole sequence of weak solutions, which converge weakly to the initial subsolution - this is referred to as a weak form of h-principle. In fluid mechanics terms the subsolution can thus be interpreted as an averaged, coarse-grained flow, with perturbations acting as fluctuations. This interpretation is explained in detail in the surveys \cite{DS12,Szn,DS17}. 

For the Cauchy problem the notion of subsolution then needs to be modified so that, at the initial time $t=0$, the subsolution already agrees with the solution. One possibility to achieve this is to first construct such a subsolution, together with its wild initial datum, by a time-restricted convex integration scheme, and then, by a second convex integration scheme pass from this subsolution to weak solutions. This ``double convex integration'' strategy, introduced in \cite{DS10} for bounded weak solutions, was first extended to H\"older spaces in \cite{Dan}. It is worth pointing out that such an extension requires substantial technical modifications, as H\"older schemes for Euler as in \cite{DS14,BDIS,Ise,BDSV} are based on rather precise estimates on the H\"older norms along the iteration sequence, whereas schemes producing bounded solutions \cite{DS09} are rather ``soft'' in comparison and can be based on an application of the Baire category theorem. We emphasize that this strategy is required to show that there exists a dense set of initial data for which the solution is non-unique. If in contrast one is only interested in proving the non-uniqueness for a single initial data, simpler strategies exist, see for instance \cite{CDD}.

In \cite{Dan} the author was able to show the existence of infinitely many $\sfrac{1}{10}^-$ H\"older initial data which are wild in the sense that to any such initial datum there exist infinitely many $\sfrac{1}{16}^-$ H\"older solutions satisfying \eqref{eq:EE}. Then, based on the uniform estimates in \cite{BDIS} for obtaining $\sfrac15^-$ weak solutions, in \cite{DanSz} the authors were able to show the statement of Theorem \ref{thm:main} above for all $\beta<1/5$. 

In this paper we adapt the technique used in \cite{DanSz} and combine with the convex integration scheme presented in \cite{BDSV} in order to prove Theorem \ref{thm:main}. In light of Onsager's conjecture this shows that (wild) non-uniqueness for the Euler equations is implied by the possibility of anomalous dissipation. 

A few words on our proof. As in \cite{DanSz} we rely on the notion of adapted subsolution, which quantifies the relationship between loss of regularity and the size of the Reynolds stress term. In order to reach any exponent $\beta<1/3$ we use the gluing technique introduced in \cite{Ise} in combination with Mikado flows, introduced in \cite{DanSz}. Although naively one might expect that the step from $\sfrac15$ to $\sfrac13$ should be a minor technical improvement, based on the improvements from the construction of $\sfrac15$-H\"older admissible weak solutions in \cite{BDIS} to $\sfrac13$-H\"older admissible weak solutions in \cite{BDSV}, there are a couple of substantial difficulties we needed to overcome. The main new challenge stems from the fact that, whilst the construction in \cite{BDIS} (used in \cite{DanSz}) is purely kinematic, making the time-localization rather straight forward, the construction in \cite{BDSV} has a crucial dynamic component (the ``gluing argument'' of Isett introduced in \cite{Ise}). This leads to the following difficulties: 
\begin{itemize}
	\item A consequence of the gluing technique of Isett in \cite{Ise} is that, along the scheme, one does not have uniform control over the energy (and the energy gap). Indeed, this lack of control of the energy profile led to the conjecture that for such weak solutions the time-regularity should generically be minimal (see \cite{IseO}). This means that in our scheme the mollification step has to be done with a time-dependent parameter.
	\item In the schemes in \cite{BDIS,DanSz} the presence of high-frequency oscillations immediately leads to the approximation result and hence to non-uniqueness. In contrast, the additional gluing step in \cite{BDSV} means that the weak solutions so obtained do not approximate in a weak norm. To overcome this problem requires introducing an additional step in passing from adapted subsolutions to weak solutions.      
\end{itemize}   
   
This paper is organized as follows: In Section \ref{sect:prel} we set the notation and recall from \cite{DanSz} the construction of the Mikado flows.
In Section \ref{sect:subs} we define the different notions of subsolutions (namely, strict, strong and adapted), we state the main Propositions allowing to approximate one concept of subsolution with another one and in the end we show how to obtain from such propositions the main Theorem \ref{thm:main}. 
In Section \ref{sect:strstr} we show how to approximate a strict subsolution with a strong subsolution.
Sections \ref{sect:gluing} and  \ref{sect:pert} contain respectively the localized gluing and localized perturbation steps needed in the double convex integration scheme. In Section \ref{sect:stradapt} we show how to obtain an adapted subsolution from a strong subsolution and in Section \ref{sect:adaptsol} how to construct solutions with the same initial datum of an adapted subsolution.

\subsection*{Acknowledgements}
L. Sz. gratefully acknowledges the support of Grant Agreement No. 724298-DIFFINCL of the European Research Council. 

\section{Preliminary results}
\label{sect:prel}

\subsection{Notation} 
Throughout this paper our spatial domain is $\T=\R^3/(2\pi\Z)^3$ the three-dimensional flat torus. 

\smallskip

We denote by $\mathcal S^{3\times 3}$ the set of symmetric $3\times3$ matrices, $\mathcal S^{3\times3}_0$ is the set of symmetric trace-free matrices,  $\mathcal S^{3\times 3} _+$ are the symmetric positive definite ones and $\mathcal S^{3\times 3} _{\geq0}$ are the symmetric positive semidefinite ones. Given a matrix $R\in\mathcal S^{3\times 3}$, we denote by $\tr R$ its trace and we often use the decomposition
\[
R=\tfrac13\tr R\,\Id+\mR=\rho\,\Id+\mR,
\]
where $\mR\in \mathcal S^{3\times 3}_0$ is the traceless part of $R$ (the projection of $R$ onto  $\mathcal S^{3\times3}_0$) and 
$\Id$ denotes the $3\times 3$ identity matrix. 

\smallskip

We recall the usual (spatial) H\"older spaces. Let $m=0,1,2,\dots$, $\alpha\in(0,1)$ and $\theta$ is a multi-index. For $f:\T\times[0,T]\to\R^3$ we denote by $\|f\|_0=\sup_{\T\times[0,T]} |f(x,t)|$. The H\"older seminorms are defined as
\begin{align*}
[f]_m&=\max_{|\theta|=m}\|D^\theta f\|_0,\\
[f]_{m+\alpha}&=\max_{\theta=m}\sup_{x\neq y,t}\frac{|D^\theta f(x,t)-D^\theta f(y,t)|}{|x-y|^\alpha},
\end{align*}
where $D^\theta=\partial_{x_1}^{\theta_1}\partial_{x_2}^{\theta_2}\partial_{x_3}^{\theta_3}$ are spatial partial derivatives. The H\"older norms are then given by
\begin{align*}
\|f\|_m=\sum_{j=0}^m[f]_j,\quad \|f\|_{m+\alpha}=\|f\|_m+[f]_{m+\alpha}.
\end{align*}
If the time-dependence is to be made explicit, we will write $[f(t)]_\alpha$, $\|f(t)\|_\alpha$, etc. 

We will use the following standard inequalities for H\"older norms:
\begin{align*}
[fg]_r&\leq C([f]_r\|g\|_0+\|f\|_0[g]_r),\\
[f]_s&\leq C\|f\|_0^{1-s/r}[f]_r^{s/r},
\end{align*}
for $0\leq s\leq r$. 
Moreover, for $f:\T\times[0,T]\to \mathcal{S}\subset \R^d$ and $\Psi:\mathcal{S}\to\R$,  for the composition we have
\begin{align*}
	[\Psi\circ f]_m&\leq C([\Psi]_1\|Df\|_{m-1}+\|D\Psi\|_{m-1}\|f\|_0^{m-1}\|f\|_m),\\
	[\Psi\circ f]_m&\leq C([\Psi]_1\|Df\|_{m-1}+\|D\Psi\|_{m-1}[f]_1^m).
\end{align*}
We also recall the following estimates on mollification.
\begin{proposition}\label{p:mollify}
Let  $\varphi\in C_c^\infty(\R^3)$ be non-negative, symmetric and such that $\int\varphi=1$. Then for any $r,s\geq 0$ we have
\begin{equation}\nonumber
\|f*\varphi_\ell\|_{r+s}\leq C\ell^{-s}\|f\|_r,
\end{equation} 
\begin{equation}\label{e:mollify2}
\|f-f*\varphi_\ell\|_r \leq C\ell^{2}\|f\|_{r+2},
\end{equation}
\begin{equation*}
\|(fg)*\varphi_\ell-(f*\varphi_\ell)(g*\varphi_\ell)\|_r\leq C\ell^{2 -r}\|f\|_1\|g\|_1.
\end{equation*}
The constant $C$ depends only on $r$ and $s$. 
\end{proposition}
Next, we recall that $H^{-1}(\T)$ is the dual space of $H^1_0(\T)$, the Sobolev space of periodic functions with average zero, with norm
\[
\|f\|_{H^{-1}}=\sup_{\|\varphi\|_{H^1_0}\leq1}\int_{\T}f\varphi\,dx.
\]

\subsection{Mikado flows}
\label{subs:mikado}

We recall Mikado flows, the basic building blocks for the convex integration scheme introduced in \cite{DanSz}.

\begin{lemma}\label{lemma:mikado}
	For any compact subset $\mathcal N\subset\subset\mathcal S^{3\times 3}_+$ there exists a smooth vector field 
	\[
	W:\mathcal N\times\T\to\R^3
	\]
	such that, for every $R\in\mathcal N$
	\begin{equation}\label{eq:mik1}
	\left\{\begin{aligned}
	\div_\xi(W(R,\xi)\otimes W(R,\xi))&=0\\
	\div_\xi W(R,\xi)&=0
	\end{aligned}\right.
	\end{equation}
	and 
	\begin{equation}\label{eq:mik2}
	\fint_{\T}W(R,\xi)\,d\xi=0,
	\end{equation}
	\begin{equation}\label{eq:mik3}
	\fint_{\T}W(R,\xi)\otimes W(R,\xi)\,d\xi=R.
	\end{equation}
\end{lemma}
 Using the fact that $W(R,\xi)$ is $\T$-periodic and has zero mean in $\xi$, we write 
 \begin{equation}\label{eq:mikFourier}
 W(R,\xi)=\sum_{k\in\Z^3\setminus\{0\}}a_k(R)A_ke^{ik\cdot\xi}
 \end{equation}
 for some coefficients $a_k(R)$ and complex vector $A_k\in\mathbb C^3$, satisfying $A_k\cdot k=0$ and $|A_k|=1$. From the smoothness of $W$ we further infer
 \begin{equation*}
 \sup_{R\in\mathcal N}|D^N_Ra_k(R)|\leq\frac{ C(\mathcal N,N,m)}{|k|^m}
 \end{equation*}
 for some constant $C$ which depends only on $\mathcal N, N$ and $m$.
 
 \begin{remark}\label{rem:Mconst}
 	The choice of $\mathcal N=B_{1/2}(\Id)$, together with the choice of $N$ and $m$ determines the constant $M$ in Proposition \ref{prop:pert}.
 \end{remark}

Using the Fourier representation we see that from \eqref{eq:mik3}
\begin{equation}
\label{eq:Ck}
W(R,\xi)\otimes W(R,\xi)=R+\sum_{k\neq0}C_k(R)e^{ik\cdot\xi}
\end{equation}
where 
\begin{equation*}
C_k k=0\quad\text{and}\quad\sup_{R\in\mathcal N}|D_R^NC_k(R)|\leq\frac{C(\mathcal N,N,m)}{|k|^m}
\end{equation*}
for any $m,N\in\N$.

\subsection{The operator $\mathcal{R}$}
\label{ss:R}

We recall also the definition of the operator $\mathcal R$ from Section 4.5 in \cite{DS13}.

\begin{definition}
	Let $v\in C^\infty(\T;\R^3)$ be a smooth vector field. We define $\mathcal Rv$ to be the matrix valued periodic function
	\begin{equation}\label{eq:Rdef}
	\mathcal Rv:=\frac14(\nabla \mathcal Pu+(\nabla \mathcal Pu)^T)+\frac34(\nabla u+(\nabla u)^T)-\frac12(\div u)\Id,
	\end{equation}
	where $u\in C^\infty(\T;\R^3)$ is the solution of 
	\[
	\triangle u=v-\fint_{\T}v\quad\text{ in }\T
	\]
	with $\int_{\T}u=0$ and $\mathcal P$ is the Leray projection onto divergence-free fields with zero average.
\end{definition}
\begin{lemma}
	For any $v\in C^\infty(\T;\R^3)$ the tensor $\mathcal Rv$ is symmetric and trace-free,   and $\div \mathcal Rv=v-\fint_{\T} v$.
\end{lemma}

The following proposition is a consequence of classical stationary phase techniques. For a detailed proof see \cite{DanSz}, Lemma 2.2. 

\begin{proposition}\label{prop:stat}
Let $\alpha\in(0,1)$ and $N\geq1$. Let $a\in C^\infty(\T)$, $\Phi\in C^\infty(\T;\R^3)$ be smooth functions and assume that 
\[
\bar C^{-1}\leq|\nabla\Phi|\leq \bar C
\]
holds on $\T$. Then
\begin{equation}\label{eq:aphiint}
\Big|\int_{\T}a(x)e^{ik\cdot\Phi}\,dx\Big|\leq C\frac{\|a\|_N+\|a\|_0\|\Phi\|_N}{|k|^N}
\end{equation}
and for the operator $\mathcal R$ defined in \eqref{eq:Rdef}, we have
\begin{equation}
\Big\|\mathcal R\Big(a(x)e^{ik\cdot\Phi}\Big)\Big\|_\alpha\leq C\frac{\|a\|_0}{|k|^{1-\alpha}}+C\frac{\|a\|_{N+\alpha}+\|a\|_0\|\Phi\|_{N+\alpha}}{|k|_N-\alpha},
\end{equation}
where the constant $C$ depends on $\bar C$, $\alpha$ and $N$ but not on $k$.
\end{proposition}

\section{Subsolutions and proofs of the main results}
\label{sect:subs}

In this section we introduce the various notions of subsolutions needed to perform the convex integration schemes, and state the main propositions which allow us to pass from one subsolutions to a stronger one. The combination of these propositions then leads to our main theorem, as in \cite{DanSz}.

The first notion of subsolution is the same as that defined in \cite{DanSz} and coincides with the notion of subsolution introduced in \cite{DS12}.
\begin{definition}
	[Strict subsolution] A subsolution is a triple 
	\[
	(v,p,R):\T\times(0,T)\to\R^3\times\R\times\mathcal S^{3\times3}
	\]
	such that $v\in L^2_{\mathrm{loc}}$, $R\in L^1_{\mathrm{loc}}$, $p$ is a distribution, the equations
	\begin{equation}\label{eq:ER}
	\begin{split}
	\partial_tv+\div (v\otimes v)+\nabla p&=-\div R\\
	\div v&=0
	\end{split}
	\end{equation}
	hold in the sense of distributions in $\T\times(0,T)$ and moreover $R\geq0$ a.e.. If $R>0$ a.e., then the subsolution is said to be strict.
\end{definition}
The next notion of subsolution is similar to the one defined in \cite{DanSz},  differing only in point \eqref{eq:RtrR}.

\begin{definition}[Strong subsolution]
	A strong subsolution with parameter $\gamma>0$ is a subsolution $(v,p,R)$ such that in addition $\tr R$ is a function of time only and, if 
	\[
	\rho(t):=\frac13\tr R,
	\]
	then 
	\begin{equation}\label{eq:RtrR}
	\Big|\mR(x,t)\Big|\leq \rho^{1+\gamma}(t)\quad\forall\,(x,t).
	\end{equation}	
\end{definition}

\begin{remark}\label{rem:gamma}
	In our schemes $\rho$ will be sufficiently small so that in particular $\rho^\gamma\leq r_0$, where $r_0$ is the geometric constant in \cite{DanSz}. Therefore \eqref{eq:RtrR} implies that our strong subsolutions satisfy Definition 3.2 in \cite{DanSz}.
Note also that if $(v,p,R)$ is a strong subsolution for some parameter $\gamma>0$, then also for any $\gamma'$ with $0<\gamma'<\gamma$.  
	\end{remark}

The next notion of subsolution has vanishing Reynolds stress at time $t=0$ and the $C^1$-norms blow up at certain rates as the Reynolds stress goes to zero. Such adapted subsolutions have been introduced in \cite{DanSz}, but this time the blow-up rate is different because it has to be consistent with a $C^{1/3-\eps}$-scheme rather than a $C^{1/5-\eps}$-scheme as in \cite{DanSz}. 

\begin{definition}[Adapted subsolution]\label{d:adapted}
	Given $\gamma>0$, $0<\beta<1/3$, and ${\nu}$ satisfying 
	\begin{equation}
	{\nu}>\frac{1-3\beta}{2\beta}
	\end{equation}
	we call a triple $(v,p,R)$ a $C^\beta$-adapted subsolution on $[0,T]$ with parameters $\gamma$ and ${\nu}$ if 
	\[
	(v,p,R)\in C^\infty(\T\times (0,T])\cap C(\T\times[0,T])
	\] 
	is a strong subsolution with parameter $\gamma$ with initial data
	\[
	v(\cdot,0)\in C^\beta(\T),\quad R(\cdot,0)\equiv 0
	\]
	and, setting $\rho(t):=\frac13 \tr R(x,t)$, for all $t>0$ we have $\rho(t)>0$ and there exists $\alpha\in (0,1)$ and $C\geq 1$ such that
	\begin{align}
\|v\|_{1+\alpha}&\leq C\rho^{-(1+{\nu})}\,,\label{eq:adaptv}\\
	|\partial_t \rho|& \leq C\rho^{-{\nu}}\,.\label{eq:adaptdtr}
	\end{align}
\end{definition}

The heuristic is as follows (see also \cite{CSz}): the Reynolds stress $R$ in the subsolution is proportional to the kinetic energy gap, so that $\rho\sim |w|^2$, where $w$ is the fluctuation, i.e.~the perturbation (obtained by convex integration) required so that $v+w$ is a solution. Therefore \eqref{eq:adaptv}, taking $\alpha=\nu=0$ for simplicity, is consistent with the scaling $|\nabla w|\lesssim |w|^{-2}$. In other words we expect $|\nabla |w|^3|\lesssim 1$.  

Our first proposition shows that one can approximate a smooth strict subsolution with an adapted subsolution.

\begin{proposition}\label{prop:stradapt}
	Let $(v,p,R)$ be a smooth strict subsolution on $[0,T]$. Then, for any $0<\beta<1/3$, $\nu>\frac{1-3\beta}{2\beta}$  and $\delta>0$ there exists $\gamma>0$ and a $C^{\beta}$-adapted subsolution $(\hat v,\hat p,\hat R)$ with parameters $\gamma,\nu$ such that $\hat\rho\leq \delta$ and 
	\begin{align*}
	\int_{\T}|\hat v|^2+\tr\hat R=\int_{\T} |v|^2+\tr R&\quad\textrm{ for all }t\in[0,T],\\
	\|v-\hat v\|_{H^{-1}}&<\delta,\\
	\|\hat v\otimes \hat v+\hat R-v\otimes v-R\|_{H^{-1}}&<\delta.
	\end{align*}
\end{proposition}

The proof will be given in Section \ref{sect:stradapt}.

Next, we show that at the small loss of the exponent $\beta$ one can approximate adapted subsolutions by weak solutions with the same initial datum. 

\begin{proposition}\label{prop:adaptsol} Let $0<\hat\beta<\beta<1/3$, $\gamma>0$, $\eta>0$ and $\nu>0$ with
	\begin{equation*}
	\frac{1-3\beta}{2\beta}<\nu<\frac{1-3\hat\beta}{2\hat\beta}.
	\end{equation*}
There exists $\delta>0$ such that the following holds. 

If $(\hat v,\hat p, \hat R)$ is a $C^{\beta}$-adapted subsolution with parameters $\gamma,\nu$ and $\hat\rho\leq\delta$, then for any $\eta>0$ there exists a $C^{\hat\beta}$-weak solution $v$ of \eqref{eq:EE} with initial datum 
	\[
	v(\cdot,0)=\hat v(\cdot,0),
	\]
	such that
	\begin{align*}
	\int_{\T}|v|^2=\int_{\T}|\hat v|^2&+\tr\hat R\quad\textrm{  for all $t\in [0,T]$,}\\
	\|v-\hat v\|_{H^{-1}}&\leq\eta,\\
	\|v\otimes v-\hat v\otimes \hat v-\hat R\|_{H^{-1}}&\leq\eta.
	\end{align*}
\end{proposition}

As a consequence, we get the following criterion for wild initial data:

\begin{corollary}\label{cor:wild}
	Let $w\in C^\beta$ be a divergence-free vectorfield for some $0<\beta<1/3$. If there exists a $C^{\hat\beta}$-adapted subsolution $(\hat v, \hat p, \hat R)$ for some $\beta<\hat{\beta}<\frac13$ with parameters $\gamma$, $\nu$ and satisfying  $\hat{\rho}\leq\delta$ as in Proposition \ref{prop:adaptsol} such that $\hat v(\cdot,0)=w(\cdot)$ and 
	\[
	\int_{\T}|\hat v(x,t)|^2+\tr\hat R(x,t)\,dx\leq\int_{\T}|w(x)|^2\,dx \quad\forall\,t>0,
	\]
	then $w$ is a wild initial datum in $C^\beta$. 
	\end{corollary}
Indeed, as observed in \cite{DanSz}, given a $C^{\hat\beta}$-adapted subsolution $(\hat v,\hat p,\hat R)$ with such parameters, Proposition \ref{prop:adaptsol} provides  a sequence of $C^\beta$ admissible weak solutions $(v_k,p_k)$ with $v_k(\cdot,0)=\hat v(\cdot,0)$,
\[
\int_{\T}|v_k(x,t)|^2\,dx=\int_{\T}|\hat v(x,t)|^2+\tr\hat R(x,t)\,dx\quad \forall\,t>0 
\]
and such that $v_k\to\hat v$ in $H^{-1}(\T)$ uniformly in time.

\begin{proof}[Proof of Theorem \ref{thm:main}]

The proof of Theorem \ref{thm:main} follows from Proposition \ref{prop:stradapt} and Corollary \ref{cor:wild} as in Section 4 of \cite{DanSz}.

\end{proof}

\section{From strict to strong subsolutions}
\label{sect:strstr}

We first state a variant of \cite{DanSz}[Proposition 3.1].

\begin{proposition}\label{prop:str}
Let $(v,p,R)$ be a smooth solution of \eqref{eq:ER} and $S$ be a smooth $\mathcal{S}^{3\times 3}$-valued matrix-field on $\T\times[0,T]$, such that one of the following two conditions is satisfied:
\begin{enumerate}
\item[(i)] $S(x,t)$ is positive definite for all $(x,t)$;
\item[(ii)] $S(x,t)=\sigma(t)\Id+\mathring{S}(x,t)$, with $|\mathring{S}|\leq \tfrac12\sigma$ for all $(x,t)$. 	
\end{enumerate}
Fix $\bar\alpha\in(0,1)$. 
Then for any $\lambda>1$ there exists a smooth solution $(\tilde v,\tilde p,\tilde R)$ with
\begin{equation}\label{eq:strong-int}
\begin{split}
  (\tilde v,\tilde p,\tilde R)&=(v,p,R)\,\textrm{ for }t\notin \supp\,\sigma\\
	\int |\tilde v|^2+\tr\tilde R=&\int |v|^2+\tr R\quad\textrm{ for all }t,
\end{split}
\end{equation}
and the following estimates hold:
\begin{equation}\label{eq:strong-est}
\begin{split}	
\|\tilde v-v\|_{H^{-1}}&\leq \frac{C}{\lambda},\\
\|\tilde v\|_{k}&\leq C\lambda^k\quad k=1,2,\\	
\|R-\tilde R-S\|_0&\leq \frac{C}{\lambda^{1-\bar\alpha}},\\
\|\tilde v\otimes\tilde v-v\otimes v+\tilde{R}-R\|_{H^{-1}}&\leq \frac{C}{\lambda^{1-\bar\alpha}}.
\end{split}
\end{equation}
Moreover, $\tr(R-\tilde R-S)$ is a function of $t$ only and satisfies
\begin{equation}\label{eq:strong-time}
\left|\frac{d}{dt}\tr(R-\tilde R-S)\right|\leq C\lambda^{\bar\alpha}. 	
\end{equation}
The constant $C\geq 1$ above depends on $(v,p,R)$, $S$ and $\bar\alpha$, but not on $\lambda$.
\end{proposition}

\begin{proof}
The proof is a minor modification of the proof given in \cite{DanSz}[Section 5]. We recall the main steps. 
Define the inverse flow of $v$, $\Phi:\T\times[0,T]\to\T$, as the solution of
\begin{equation*}
\left\{
\begin{aligned}
\partial_t\Phi+v\cdot\nabla\Phi&=0\\
\Phi(x,0)&=x,\quad\forall\,x\in\T
\end{aligned}\right.
\end{equation*}  
and set 
\[
\bar R(x,t)=\begin{cases}D\Phi(x,t)S(x,t)D\Phi^T(x,t)&\textrm{ if (i) holds;}\\ D\Phi(x,t)\frac{\mathring{S}(x,t)}{\sigma(t)}D\Phi^T(x,t)&\textrm{ if (ii) holds.}\end{cases}
\] 
Observe that in case (i) $\bar R$ is defined on $\T\times [0,T]$ and, being continuous and defined on a compact set, takes values in a compact subset $\mathcal N_0$ of $S^{3\times3}_+$. In case (ii) $\bar R$ is defined only on $\T\times\supp\,\sigma$, and takes values in $\mathcal N_0:=B_{1/2}(\Id)$. 
 
By Lemma \ref{lemma:mikado}, there exists a smooth vectorfield 
$W:\mathcal N_0\times\T\to\R^3$ with properties \eqref{eq:mik1}-\eqref{eq:mik3}.
We define 
\begin{align*}
w_o(x,t)&=\begin{cases}D\Phi^{-1}W(\bar R,\lambda\Phi(x,t))&\textrm{ if (i) holds;}\\
\sigma^{1/2}D\Phi^{-1}W(\bar R,\lambda\Phi(x,t))&\textrm{ if (ii) holds;}\end{cases}\\
w_c(x,t)&=\begin{cases}-\frac{1}{\lambda}\curl (D\Phi^TU(\bar R,\lambda\Phi(x,t)))-w_o&\textrm{ if (i) holds;}\\
-\frac{1}{\lambda}\curl (\sigma^{1/2}D\Phi^TU(\bar R,\lambda\Phi(x,t)))-w_o&\textrm{ if (ii) holds.}\end{cases}
\end{align*}
Here $U=U(S,\xi)$ is such that $\curl_\xi U=W$. We then define
\begin{equation*}
\tilde v=v+w_o+w_c,\quad\tilde p=p+\bar p,\quad \tilde R=R-S-\mathring{\mathcal{E}}^{(1)}-\mathcal{E}^{(2)},
\end{equation*}
where $\bar p=-\tfrac{1}{3}(w_c\cdot \tilde v+w_o\cdot w_c)$, 
\begin{align*}
\mathring{\mathcal{E}}^{(1)}&=\mathcal R(F)+(w_c\otimes\tilde v+w_o\otimes w_c+\bar p\Id),\\
F&=\div(w_o\otimes w_o-S)+(\partial_t+v\cdot\nabla)w_o+(w_o+w_c)\cdot\nabla v+\partial_t w_c,\\
\mathcal{E}^{(2)}&=\frac13\Big(\fint_{\T}|\tilde v|^2-|v|^2-\tr S\Big)\Id
\end{align*}
and $\mathcal R$ is the operator defined in \eqref{eq:Rdef}. 

By construction \eqref{eq:strong-int} holds, $\tr\mathring{\mathcal E}^{(1)}=0$, $\mathcal{E}^{(2)}$ is a function of $t$ only, and 
\begin{align*}
\div\mathring{\mathcal E}^{(1)}&=\div(\tilde v\otimes \tilde v-v\otimes v-S)+\bar p\Id+\partial_t(\tilde v-v)\\
&=\partial_t\tilde v+\div(\tilde v\otimes \tilde v-S+R)+\tilde p\,\Id.
\end{align*}
Therefore $(\tilde v,\tilde p,\tilde R)$ solves \eqref{eq:ER} as claimed. The estimates in the proof of \cite{DanSz}[Proposition 3.1] apply to $\mathring{\mathcal{E}}^{(1)}$ and yield then \eqref{eq:strong-est}. 

Finally, note that $\tr(R-\tilde R-S)=\tr\mathcal{E}^{(2)}=\fint|\tilde v|^2-|v|^2-\tr S$. In order to estimate $\int|\tilde{v}|^2\,dx$, note that the energy identity for $\tilde v$, deduced from \eqref{eq:ER}, reads
$$
\partial_t\tfrac12|\tilde v|^2+\div(\tilde{v}(|\tilde v|^2/2+\tilde{p})=-\tilde v\cdot\div \tilde R,
$$ 
from which we deduce, after integrating in $x$ and using \eqref{eq:strong-est}
$$
\left|\frac{d}{dt}\fint\tfrac12|\tilde v|^2\,dx\right|\leq \fint|\nabla\tilde v||\mathring{\tilde R}|\,dx\leq C\lambda^{\bar\alpha}.
$$
This verifies \eqref{eq:strong-time} and thus concludes the proof.
\end{proof}

We will use this proposition in two situations, as described in the following corollaries.

\begin{corollary}\label{cor:strstr}
	Let $(v,p,R)$ be a smooth strict subsolution on $[0,T]$ and let $\tilde\eps>0$. There exists $\tilde \delta,\,\gamma>0$ such that the following holds.
	 
For any $0<\delta<\tilde\delta$ 
	there exists a smooth strong subsolution $(\tilde v, \tilde p, \tilde R)$ with $\tilde R(x,t)=\tilde\rho(t)\Id+\mathring{\tilde{R}}(x,t)$ such that, for all $t\in[0,T]$
	\begin{align}
	\tfrac{3}{4}\delta\leq &\tilde\rho\leq \tfrac{5}{4}\delta,\label{eq:str1}\\
	|\mathring{\tilde{R}}|&\leq \tilde{\rho}^{1+\gamma}\,,\label{eq:str2}\\
		\|\tilde v-v\|_{H^{-1}}+\|v\otimes v+R-\tilde v\otimes \tilde v&-\tilde R\|_{H^{-1}}\leq C\delta^{1+\gamma}\,,\label{eq:str4}\\
		\int_{\T}|v|^2+\tr R\,dx&=\int_{\T}|\tilde v|^2+\tr\tilde R\,dx,\label{eq:str5}\\
		\|\tilde v\|_j&\leq C\delta^{-j(1+\tilde\eps)}\quad j=1,2,\label{eq:str3}\\
		\left|\partial_t\tilde\rho\right|&\leq C\delta^{-\tilde\eps}\,,\label{eq:str6}
	\end{align}
	where the constant $C$ depends on $(v,p,R)$ and $\tilde\eps$.
\end{corollary}

\begin{proof}[Proof of Corollary \ref{cor:strstr}]
Let
\[
\tilde\delta=\tfrac12\inf\{R(x,t)\xi\cdot\xi:\,|\xi|=1,\,x\in\T,t\in[0,T]\}.
\]
Since $R$ is a smooth positive definite tensor on a compact set, $\tilde\delta>0$. Then $S:=R-\delta\Id$ is positive definite for any $\delta<\tilde\delta$. We may in addition assume without loss of generality that $\delta\leq 1$. We apply Proposition \ref{prop:str} with $(v,p,R)$, $S$, and $\bar\alpha\in(0,1)$ to be chosen below. Note that condition (i) is satisfied. The proposition yields a smooth solution $(\tilde v,\tilde p,\tilde R)$ of \eqref{eq:ER} with properties \eqref{eq:strong-int}-\eqref{eq:strong-time}. Observe that $\tilde R-R+S=\tilde R-\delta\,\Id$, so that $\tilde\rho=\tfrac{1}{3}\tr(\tilde R-R+S)+\delta$ is a function of $t$ only.   

For $\gamma\in(0,1)$ (to be specified later) set
$$
\lambda=(4C)^{\frac{1}{1-\bar\alpha}}\delta^{-\frac{1+\gamma}{1-\bar\alpha}}
$$
with the constant $C$ from \eqref{eq:strong-est}, so that we obtain \eqref{eq:str4} and 
$$
\|\tilde R-R+S\|_0\leq \tfrac{1}{4}\delta^{1+\gamma}\,.
$$
It follows that $|\tilde\rho-\delta|\leq \tfrac{1}{4}\delta$, verifying \eqref{eq:str1}. From this estimate we can in turn deduce \eqref{eq:str2}.

So far $\bar{\alpha},\gamma$ was arbitrary - it remains to choose these parameters so that also \eqref{eq:str3} and \eqref{eq:str6} are valid. Indeed, by choosing $0<\bar\alpha,\gamma\ll 1$ sufficiently small, so that $\tfrac{1+\gamma}{1-\bar\alpha}<1+\tilde\eps$ and $\bar\alpha\tfrac{1+\gamma}{1-\bar\alpha}<\tilde\eps$, we easily deduce \eqref{eq:str3} and \eqref{eq:str6}. 

\end{proof}

\begin{corollary}\label{cor:adapt}
Given $0<\beta<1/3$ and $\gamma,\nu>0$ there exists $\tilde\delta>0$ such that the following holds. 

Let $(v,p,R)$ be a $C^{\beta}$-adapted subsolution with parameters $\gamma,\nu>0$ and assume $\rho\leq\tilde{\delta}$. Suppose $\gamma<\nu$ and let $\tilde\gamma<\gamma$. For any $\eta>0$ there exists another $C^{\beta}$-adapted subsolution $(\tilde v,\tilde p,\tilde R)$ with parameters $\tilde\gamma,\nu>0$ (with possibly different constants $C$ and $\alpha$ in \eqref{eq:adaptv}-\eqref{eq:adaptdtr} which may depend on $(v,p,R)$ but not on $\eta$) such that, with $\tilde R=\tilde\rho\,\Id+\mathring{\tilde R}$, 
$$
\tilde\rho\leq \eta\quad \textrm{ and }\quad\tilde v=v\textrm{ for }t=0.
$$	
Furthermore
\begin{equation}\label{eq:approxadapt}
\begin{split}
\int_{\T}|\tilde v|^2+\tr \tilde R=\int_{\T}|v|^2+\tr R&\quad\textrm{ for all }t,\\
\|\tilde v-v\|_{H^{-1}}&\leq \eta,\\
\|\tilde v\otimes \tilde{v}+\tilde R-v\otimes v-R\|_{H^{-1}}&\leq \eta.
\end{split}
\end{equation}
\end{corollary}

\begin{proof}[Proof of Corollary \ref{cor:adapt}]
Set $\tilde{\delta}=4^{-1/\gamma}$ and assume $(v,p,R)$ be a $C^{\beta}$-adapted subsolution satisfying \eqref{eq:adaptv}-\eqref{eq:adaptdtr} with parameters $\gamma,\nu>0$, such that $\rho\leq\tilde\delta$. Then $\rho^{\gamma}\leq \tfrac14$. We may assume moreover, that $\eta\leq\tilde\delta$. 

	Let $\phi\in C_c^{\infty}(0,\infty)$ be a cut-off function such that $\phi(s)=1$ for $s\geq 1/2$, $\phi(s)=0$ for $s\leq 1/4$, and set 
	$$
	\psi(t)=\phi\left(\frac{\rho(t)}{\eta}\right).
	$$ 
	Then, using the bound on $\partial_t\rho$ from \eqref{eq:adaptdtr} we deduce $|\partial_t\psi|\leq C\eta^{-(1+\nu)}$. Here and in the subsequent proof we denote by $C$ generic constants which may depend on $(v,p,R)$. Define $S=\psi(R-\frac{\eta}{8}\Id)$. Then $S=\sigma\,\Id+\mathring{S}$, with
	$$
	\sigma(t)=\psi(t)\big(\rho(t)-\frac{\eta}{8}\big)\geq \frac{1}{2}\psi\rho,
	$$
	since $\rho\geq \eta/4$ on $\supp\,\psi$. Moreover, on $\supp\,\psi$
	$$
	|\mathring{S}|=|\psi\mathring{R}| \leq \psi\rho^{1+\gamma}\leq 2\rho^\gamma\sigma\leq \frac{1}{2}\sigma.
	$$
	Thus condition (ii) in Proposition \ref{prop:str} for $S$ is satisfied. 
	
	We apply the proposition with $\bar\alpha>0$, $\lambda\geq 1$ to be chosen below and obtain a smooth solution $(\tilde v,\tilde p,\tilde R)$ of \eqref{eq:ER} with properties \eqref{eq:strong-int}-\eqref{eq:strong-time}. In particular we obtain 
	\begin{align*}
	\tilde{R}&=R-S-\mathcal{E}=(1-\psi)R+\frac{\psi\eta}{8}\Id-\mathcal{E},\\
	\tilde{\rho}&=(1-\psi)\rho+\frac{\psi\eta}{8}-\frac{1}{3}\tr\mathcal{E},
	\end{align*}
	where $\|\mathcal E\|_0\leq C\lambda^{-1+\bar\alpha}$. Choose
	$$
	\lambda=(4C)^{\frac{1}{1-\bar\alpha}}\eta^{-\frac{1+\gamma}{1-\bar\alpha}},
	$$
	so that $\|\mathcal{E}\|_0\leq \frac14\eta^{1+\gamma}\leq \frac{1}{16}\eta$. Then, 
	observing that $\rho\geq \eta/4$ on $\supp\psi$, we deduce 
	$$
	\tilde\rho\geq (1-\psi)\frac{\eta}{4}+\psi\frac{\eta}{8}-\frac{\eta}{16}\geq  \frac{\eta}{16}\textrm{ on }\supp\psi,
	$$
	whereas $\tilde\rho=\rho$ otherwise. Furthermore, since $\psi=1$ if $\rho\geq \eta/2$, 
	$$
	\tilde\rho\leq (1-\psi)\frac{\eta}{2}+\psi\frac{\eta}{8}+\frac{\eta}{16}\leq \eta.
	$$
	Similarly, on $\supp\psi$  
	$$
	|\mathring{\tilde{R}}|\leq (1-\psi)|\mathring{R}|+\frac14\eta^{1+\gamma}\leq \frac12\eta^{1+\gamma}+\frac14\eta^{1+\gamma}\leq \tilde C\tilde\rho^{1+\gamma}.
	$$
	Thus, by choosing $\eta>0$ sufficiently small (such that $\eta^{\gamma-\tilde\gamma}<1/\tilde C$), we obtain $|\mathring{\tilde{R}}|\leq \tilde{\rho}^{1+\tilde\gamma}$, so that $(\tilde v,\tilde p,\tilde R)$ is a strong subsolution with parameter $\tilde\gamma$. 
Moreover, it is easy to see that \eqref{eq:approxadapt} holds. It remains to verify \eqref{eq:adaptv}-\eqref{eq:adaptdtr}. Since $\tilde{v}=v$ and $\tilde{\rho}=\rho$ outside $\supp\psi$, in the following we restrict to times $t\in \supp\psi$. 

From \eqref{eq:strong-est} and interpolation we obtain for any $\alpha\in[0,1]$
\begin{align*}
\|\tilde v\|_{1+\alpha}\leq \eta^{{-(1+\alpha)}\frac{1+\gamma}{1-\bar\alpha}},\quad |\partial_t\tr\mathcal{E}|\leq \eta^{-\bar{\alpha}\frac{1+\gamma}{1-\bar\alpha}},
\end{align*}
whereas from the definition of $\tilde\rho$ we have that 
\begin{align*}
|\partial_t\tilde\rho|\leq |\partial_t\rho|+|\partial_t\psi|\eta+|\partial_t\tr\mathcal{E}|\leq C(1+\eta^{-\nu}+\eta^{-\bar{\alpha}\frac{1+\gamma}{1-\bar\alpha}}).
\end{align*}
Therefore  \eqref{eq:adaptv}-\eqref{eq:adaptdtr} holds with constant $C$ and $\alpha>0$ provided 
$$
 (1+\alpha)\frac{1+\gamma}{1-\bar\alpha}<1+\nu,\quad \bar\alpha\frac{1+\gamma}{1-\bar\alpha}<\nu.
 $$
 Both inequalities can be satisfied by choosing $\bar\alpha,\alpha>0$ sufficiently small, provided $\gamma<\nu$. This concludes the proof. 
	\end{proof}

\section{Guide to the subsequent sections}\label{s:guide}

Let us briefly recall the  convex integration scheme in \cite{BDSV}, in which an approximating sequence 
$(v_q,p_q,R_q)$ of subsolutions is constructed. The various $C^0$ and $C^1$ norms of the subsolution are controlled in terms of parameters $\delta_q,\lambda_q$, where we can think of $\delta_q^{1/2}$ as an amplitude and $\lambda_q$ as a (spatial) frequency. 
This sequence of parameters is defined as
\begin{equation}\label{eq:dl2}
\lambda_q=2\pi [a^{b^q}],\qquad\delta_q=\lambda_q^{-2\beta},
\end{equation}
where
\begin{itemize}
	\item $[x]$ denotes the smallest integer $n\geq x$.
\item $\beta\in (0,1/3)$ and $b\in (1,3/2)$ control the H\"older exponent of the scheme and are required to satisfy
\begin{equation}\label{eq:bbeta}
	1<b<\frac{1-\beta}{2\beta};
	\end{equation}
\item $a\gg 1$ is chosen sufficiently large in the course of the proofs (in order to absorb various constants in the estimates).
\end{itemize}
In \cite{BDSV} the stage $q\mapsto q+1$ amounts to the statement that there exists a universal constant $M>1$ such that for $0<\alpha$ sufficiently small and sufficiently large $a\gg 1$ the following holds: given $(v_q,p_q,R_q)$ satisfying \eqref{eq:ER} and satisfying the estimates
\begin{align}
\|\mR_q\|_{0}&\leq\delta_{q+1}\lambda_q^{-3\alpha}\label{eq:2.11}\\
\|v_q\|_1&\leq M\delta_q^{1/2}\lambda_q\\
\delta_{q+1}\lambda_q^{-\alpha}&\leq\frac13\tr R_q(t)\leq\delta_{q+1},\label{eq:2.14}
\end{align} 
then there exists a solution $(v_{q+1},p_{q+1},R_{q+1})$ to \eqref{eq:ER} satisfying \eqref{eq:2.11}-\eqref{eq:2.14} with $q$ replaced by $q+1$. Moreover, we have
\[
\|v_{q+1}-v_q\|_0+\frac{1}{\lambda_{q+1}}\|v_{q+1}-v_q\|_1\leq M\delta_{q+1}^{1/2}.
\]

The proof of this statement consists of three steps:
\begin{itemize}
	\item[(1)] Mollification: $(v_q,R_q)\mapsto(v_\ell,R_\ell)$;\\
	\item[(2)] Gluing: $(v_\ell,R_\ell)\mapsto(\bar v_q,\bar R_q)$;\\
	\item[(3)] Perturbation $(\bar v_q,\bar R_q)\mapsto(v_{q+1},R_{q+1})$. 
\end{itemize}

In Section \ref{sect:gluing} we prove a localized (in time) version of the first two stages, and in Section \ref{sect:pert} a localized version of the perturbation.
We recall that the gluing stage, first introduced in \cite{Ise}, is needed in order to produce a Reynolds stress $\mathring{ \bar R}_q$ which has support in pairwise disjoint temporal regions of some suitable length in time, which is necessary in order to define perturbations through Mikado flows. 

In the sequel  we work with a sequence $(\lambda_q,\delta_q)$, $q=0,1,2,\dots$. Moreover, we fix $\alpha>0$, $\gamma>0$ and define
\begin{equation}\label{eq:ellqn}
\ell_{q}=\frac{\delta_{q+2}^{(1+\gamma)/2}}{\delta_{q}^{1/2}\lambda_{q}\lambda_{q+1}^{3\alpha/2}},
\end{equation} 
and
\begin{equation}\label{eq:tauqk}
\tau_{q}=\frac{\ell_{q}^{4\alpha}}{\delta_{q}^{1/2}\lambda_q}.
\end{equation}
As in \cite{BDSV}, we will require several inequalities between these parameters. First of all, we assume
\begin{equation}\label{e:lambdadelta}
	\frac{\delta_{q+1}^{1/2}\delta_q^{1/2}\lambda_q}{\lambda_{q+1}^{1-8\alpha}}\leq \delta_{q+2}.
\end{equation}
To verify this, we use \eqref{eq:dl2} and take logarithm base $\lambda_q$ to see that \eqref{e:lambdadelta} follows for sufficiently large $a\gg 1$ provided
$$
(b-1)\big[1-\beta(1+2b)]>8\alpha b.
$$
Thus, after fixing $b,\beta$ as in \eqref{eq:bbeta}, \eqref{e:lambdadelta} will be valid for sufficiently small $\alpha>0$ (depending on $b,\beta$).  
Next, we assume
\begin{equation}
\lambda_{q+1}^{-1}\leq\ell_{q}\leq\lambda_{q}^{-1}. \label{e:lambdaell}
\end{equation} 
The second inequality is immediate from the definition. Concerning the first, 
as in \cite{BDSV} we will in fact need the following sharpening: there exists $\overline{N}\in \N$ such that
\begin{equation}\label{e:lambdaellN}
	\lambda_{q+1}^{1-\overline{N}}\leq \ell_q^{\overline{N}+1}.
\end{equation}
To verify \eqref{e:lambdaellN} we argue as above: use \eqref{eq:dl2} and \eqref{eq:ellqn} and take logarithm base $\lambda_q$ to see that \eqref{e:lambdaellN} follows for sufficiently large $a\gg 1$ provided
$$
\overline{N}\big[(b-1)(1-\beta(b+1))-\gamma\beta b^2-\tfrac32\alpha b\big]>1+b+(1+\gamma)\beta b^2+\tfrac32\alpha b-\beta.
$$
It is easy to see that this inequality is valid, provided we choose (in this order):
\begin{itemize}
\item $b,\beta$ as in \eqref{eq:bbeta}, so that in particular $\beta(1+b)<1$;
\item $0<\alpha,\gamma$ are sufficiently small depending on $b,\beta$;
\item $\overline{N}\in\N$ sufficiently large, depending on $b,\beta,\alpha,\gamma$.	
\end{itemize}
In the following sections we will use the symbol $A\lesssim B$ to denote $A\leq CB$, where $C$ is a constant whose value may change from line to line, but only depends on the universal constant $M$, on the parameters $b,\beta,\alpha,\gamma$ chosen as above, and, if norms depending on $N\in \N$ are involved, also on $N$. In particular, $C$ will never depend on the choice of $a\gg 1$.

\section{Localized gluing step}\label{sect:gluing}

The aim of this section is to prove a time-localized version of the gluing procedure of Sections 3 and 4 in \cite{BDSV}:  on intervals $[T_1,T_2]\subset[0,T]$ instead of on the whole interval $[0,T]$. The main proposition is Proposition \ref{prop:gluing}, which combines the mollification and gluing steps indicated in Section \ref{s:guide}

In the the statement of Proposition \ref{prop:gluing}, we will need the following definitions.  

\begin{definition}\label{def:intervals}
Let $0\leq T_1<T_2\leq T$ such that $T_2-T_1>4\tau_q$. We define sequences of intervals $\{I_i\}$ and $\{J_i\}$ as follows. 
	Let 
	\begin{equation}\label{e:defIi}
	t_i=i\tau_q,\quad I_i=\Big[t_i+\frac13\tau_q, t_i+\frac23\tau_q\Big]\cap [0,T],
	\end{equation}
	and let 
	\begin{equation}\label{e:defn+-}
	\begin{split}
	\underline{n}&=\begin{cases}\min\left\{i:\,t_i-\tfrac23\tau_q\geq T_1\right\}&\textrm{ if }T_1>0\\ 0&\textrm{ if }T_1=0,\end{cases}	\\
	\overline{n}&=\max\Bigl\{i:\,t_i+\tfrac23\tau_q\leq T_2\Bigr\}.
	\end{split}
	\end{equation}
Moreover, define 
	\begin{equation*}
	J_i=\Big(t_i-\frac13\tau_{q},t_i+\frac13\tau_{q}\Big)\cap [0,T], \quad \underline{n}\leq i\leq \overline{n}\,,
	\end{equation*} 	
and
\begin{equation*}
J_{\underline{n}-1}=[0,t_{\underline{n}}-\frac23\tau_{q}),\quad J_{\overline{n}+1}=(t_{\overline{n}}+\frac23\tau_{q},T].
\end{equation*}
\end{definition}
Note that
\begin{equation}\label{e:decomposition}
[0,T]=J_{\underline{n}-1}\cup I_{\underline{n}-1}\cup \Bigl[J_{\underline{n}}\cup\dots \cup J_{\overline{n}}\Bigr]\cup I_{\overline{n}}\cup J_{\overline{n}+1}
\end{equation}
is a pairwise disjoint decomposition into intervals and
\begin{equation}\label{eq:tntn}
t_{\underline{n}}<T_1+\tfrac53\tau_q<T_2-\tfrac53\tau_q<t_{\overline{n}}.	
\end{equation}
Observe also that $\underline{n}\geq 1$ if $T_1>0$, whereas $\underline{n}=0$ and $J_{\underline{n}-1}\cup I_{\underline{n}-1}=\emptyset$ if $T_1=0$. 
In the following we denote, as usual, for $R$ whose trace depends only on time,
\[
R(x,t)=\rho(t)\Id+\mR(x,t).
\]

\begin{proposition}[Localized gluing step]\label{prop:gluing} 
 Let $b,\beta,\alpha,\gamma$ and $(\delta_q,\lambda_q,\ell_{q},\tau_{q})$ be as in Section \ref{s:guide} with $\alpha/\gamma<\beta/b$.  Let $[T_1,T_2]\subset[0,T]$ with $|T_2-T_1|>4\tau_q$. Let $(v_q,p_q,R_q)$ be a strong subsolution on $[0,T]$ which on $[T_1, T_2]$ satisfies the estimates
	\begin{align}
		\tfrac34\delta_{q+2}&\leq\rho_q\leq\tfrac72\delta_{q+1}\,,\label{eq:rhoest}\\
	\|\mR_q\|_0&\leq \rho_q^{1+\gamma}\,,\label{eq:Rq}\\
	\|v_q\|_{1+\alpha}&\leq M\delta_q^{1/2}\lambda_q^{1+\alpha}\,,\label{eq:vq1}\\
   |\partial_t\rho_q|&\leq\rho_q\delta_q^{1/2}\lambda_q\,,\label{eq:dtroq}
	\end{align} 
	with some constant $M>0$.
	Then, provided $a\gg 1$ is sufficiently large, there exists $(\bar v_q,\bar p_q,\bar R_q)$ solution of \eqref{eq:ER} on $[0,T]$ such that 
	\begin{equation}\label{eq:eq1}
	(\bar v_q,\bar p_q,\bar R_q)=(v_q,p_q,R_q) \quad \text{on $[0,T]\setminus[T_1,T_2]$},
	\end{equation}
	and on $[T_1,T_2]$ the following estimates hold:
	\begin{align}	
	\|\bar v_q-v_q\|_\alpha&\lesssim\bar\rho_q^{(1+\gamma)/2}\ell_{q}^{\alpha/3}\,,\label{eq:eq2}\\
	\|\bar v_q\|_{1+\alpha}&\lesssim\delta_q^{1/2}\lambda _q^{1+\alpha}\,,\label{eq:eq3}\\
	\|\mathring{\bar R}_q\|_0&\lesssim\bar\rho_q^{1+{\gamma}}\ell_{q}^{-\alpha}\,,\label{eq:eq4bis}\\
	\tfrac78\rho_q&\leq \bar\rho_q\leq \tfrac98\rho_q\,,\label{eq:eq4-1}\\ 
	|\partial_t\bar{\rho}_q|&\lesssim\bar\rho_q\delta_q^{1/2}\lambda_q\,, \label{eq:last0}
	\end{align}
	and
	\begin{equation}\label{eq:eq4}
\left|\int_{\T}|v_q|^2-|\bar{v}_q|^2\,dx\right|\lesssim \bar\rho_q^{1+\gamma}\ell_q^{2\alpha}\,.	
	\end{equation}
	Moreover, on $[t_{\underline{n}},t_{\overline{n}}]$ the additional estimates
	\begin{align}
	\|\bar v_q\|_{N+1+\alpha}&\lesssim\delta_q^{1/2}\lambda_q^{1+\alpha}\ell_{q}^{-N}\,,\label{eq:estb1}\\	
	\Big\|{\mathring{\bar R}_q}\Big\|_{N+\alpha}&\lesssim\bar\rho_q^{1+{\gamma}}\ell_{q}^{-N-\alpha}\,,\label{eq:Rnalpha2}\\
	\|(\partial_t+\bar v_q\cdot\nabla)\mathring{\bar R}_q\|_{N+\alpha}&\lesssim\bar\rho_q^{1+{\gamma}}\delta_q^{1/2}\lambda_q\ell_{q}^{-N-5\alpha}\label{eq:last12}
	\end{align}
	hold for any $N\geq 0$. Finally, 
	\begin{equation}\label{e:gluingsupport}
	\mathring{\bar{R}}_q\equiv 0\quad\textrm{ for }t\in \bigcup_{i=\underline{n}}^{\overline{n}}J_i.
	\end{equation}
	\end{proposition}

\bigskip

\begin{proof}[Proof of Proposition \ref{prop:gluing}]\hfill

	The proof of Proposition \ref{prop:gluing} follows closely the gluing procedure \cite{BDSV}[Sections 3 and 4], with two main differences. One is  that the subsolution has to be changed only inside the interval $[T_1,T_2]$ and stay unchanged outside $[T_1,T_2]$. More precisely, recalling the decomposition \eqref{e:decomposition}, 
	\begin{itemize}
	\item the gluing procedure as in \cite{BDSV} will be performed in the interval
	\begin{equation}\label{e:gluingregion}
	\Bigl[J_{\underline{n}}\cup\dots \cup J_{\overline{n}}\Bigr]=\Bigl(t_{\underline{n}}-\tfrac13\tau_q,t_{\overline{n}}+\tfrac13\tau_q\Bigr);
	\end{equation}
	\item the subsolution will remain unchanged in $J_{\underline{n}-1}\cup J_{\overline{n}+1}$;
	\item the intervals $I_{\underline{n}-1}$ and $I_{\overline{n}}$ will be cutoff regions between the ``glued'' and ``unglued'' subsolutions.	
	\end{itemize}
The other one is that, since the trace part of $R_q$, namely $\rho_q$, has different lower and upper bounds on $[T_1,T_2]$ (respectively of the order $\delta_{q+2}$ and $\delta_{q+1}$), one needs to mollify with different parameters $\ell_{q,i}$ depending on $\rho_q(t_i)$ on $\tau_q$-neighbourhoods of the points $\{t_i\}$.

	\subsection*{Step 1 - Mollification}
	
	For all $\underline{n}\leq i\leq \overline{n}$, define
	\[
	\rho_{q,i}=\rho_q(t_i),\quad \ell_{q,i}=\frac{\rho_{q,i}^{(1+\gamma)/2}}{\delta_q^{1/2}\lambda_q^{1+3\alpha/2}}.
	\]
	Using \eqref{eq:rhoest} and assuming $a\gg 1$ is sufficiently large (as in \eqref{e:lambdaell}, depending on $\alpha,\gamma,b$) we may ensure that
	\begin{equation}\label{e:lambdaelli}
	\lambda_{q+1}^{-1}\leq \ell_q\leq \ell_{q,i}\leq\lambda_q^{-1}.
	\end{equation}
	Let $\phi$ be a standard mollification kernel in space and define
	\begin{align*}
	v_{\ell_{q,i}}&:=v_q\ast\phi_{\ell_{q,i}}\,,\\
	p_{\ell_{q,i}}&:=p_q\ast\phi_{\ell_{q,i}}+\tfrac{1}{3}(|v_q|^2\ast\phi_{\ell_{q,i}}-|v_{\ell_{q,i}}|^2)\,,\\
	\mR_{\ell_{q,i}}&:=\mR_q\ast\phi_{\ell_{q,i}}+(v_q{\mathring\otimes} v_q)\ast\phi_{\ell_{q,i}}-v_{\ell_{q,i}}{\mathring\otimes}v_{\ell_{q,i}}.
	\end{align*}
	Observe that with this definition the triple $(v_{\ell_{q,i}},p_{\ell_{q,i}},\mR_{\ell_{q,i}})$ is a solution of \eqref{eq:ER}. Using the estimates \eqref{eq:Rq}-\eqref{eq:vq1} together with the mollification estimates in Proposition \ref{p:mollify} and the choice of the mollification parameters 
	we deduce as in \cite[Proposition 2.2]{BDSV}:
	\begin{align}
	\|v_{\ell_{q,i}}-v_q\|_{\alpha}&\lesssim \delta_q^{1/2}\lambda_q^{1+\alpha}\ell_{q,i}\lesssim \rho_{q,i}^{(1+\gamma)/2}\ell_q^{\alpha/3}\,,\label{eq:vlvq1}\\
	\|v_{\ell_{q,i}}\|_{N+1+\alpha}&\lesssim\delta_q^{1/2}\lambda_q^{1+\alpha}\ell_{q,i}^{-N}\,,\label{eq:indata}\\
	\|\mR_{\ell_{q,i}}\|_{N+\alpha}&\lesssim\rho_q^{1+\gamma}\ell_{q,i}^{-N-\alpha}+\delta_q\lambda_q^{2+2\alpha}\ell_{q,i}^{2-N-\alpha}\nonumber\\
	&\lesssim \rho_q^{1+\gamma}\ell_{q}^{-N-\alpha}+\rho_{q,i}^{1+\gamma}\ell_{q}^{-N-\alpha}\,,\label{eq:Rln}\\
	\Big|\int_{\T}|v_q|^2-|v_{\ell_{q,i}}|^2\Big|&\lesssim \delta_q\lambda_q^{2+2\alpha}\ell_{q,i}^2=\rho_{q,i}^{1+\gamma}\lambda_q^{-\alpha}\,.\label{eq:vivlnorm}
	\end{align}

\subsection*{Step 2 - Gluing procedure}

Let $\{I_i\}_{\underline{n}\leq i\leq \overline{n}}$ be the sequence of intervals corresponding to $[T_1,T_2]$ according to Definition \ref{def:intervals}, 
	We define now a partition of unity on $[0,T]$ 
	\begin{equation*}
	\sum_{i=\underline{n}-1}^{\overline{n}+1}\chi_i\equiv 1
	\end{equation*}
	subordinate to the decomposition in \eqref{e:decomposition}. More precisely, for each $\underline{n}-1\leq i\leq \overline{n}+1$ the function $\chi_i\geq 0$ satisfies
	\begin{align*}
	\supp\chi_i&\subset I_{i-1}\cup J_i\cup I_i\,,\\
	\chi_i(t)&=1\text{ for $t\in J_i$}\,,\\
	|\partial_t^N\chi_{i}|&\lesssim\tau_{q}^{-N}\textrm{ for all }N\geq 0.
	\end{align*}
	We define
	\begin{equation}\label{e:defbarvq}
	\bar v_q=\sum_{i=\underline{n}-1}^{\overline{n}+1}\chi_iv_i,\quad \bar p_q^{(1)}=\sum_{i=\underline{n}-1}^{\overline{n}+1}\chi_ip_i,
	\end{equation}
where $(v_i,p_i)$ is defined as follows. For $\underline{n}\leq i\leq \overline{n}$ we define $(v_i,p_i)$ as the solution of
	\begin{equation}\label{eq:classic}
	\left\{\begin{aligned}
	&\partial_tv_i +\div(v_i\otimes v_i)+\nabla p_i=0\,,\\
	&\div v_i=0\,,\\
	&v_i(\cdot,t_i)=v_{\ell_{q,i}}(\cdot,t_i),
	\end{aligned}\right.
	\end{equation}
and set $(v_i,p_i)=(v_q,p_q)$ for $i\in \{\overline{n}+1,\underline{n}-1\}$.
Thus, we note first of all that $\div \bar v_q=0$ and moreover
$$
(\bar{v}_q,\bar p_q)=(v_q,p_q)\quad\textrm{ for }t\in [0,T]\setminus [T_1,T_2].
$$
	
Next, we define $\bar{R}_q$.  As in Section 4.1 of \cite{BDSV}, for $t\in I_i\cup J_{i+1}$ we have $\chi_i+\chi_{i+1}=1$ and therefore 
	\begin{equation*}
	\begin{split}
	\partial_t\bar v_q+\div(\bar v_q&\otimes \bar v_q)+\nabla \bar p_q=\\
	=&\partial_t\chi_i(v_i-v_{i+1})-\chi_i(1-\chi_i)\div((v_i-v_{i+1})\otimes(v_i-v_{i+1}))\\
	&-\div(\chi_iR_i+(1-\chi_i)R_{i+1}),
	\end{split}
	\end{equation*}
	where we wrote $R_i=0$ for $\underline{n}\leq i\leq \overline{n}$ and $R_i=R_q$ otherwise.
Thus, recalling the operator $\mathcal R$ defined in Proposition 4.1 \cite{BDSV} (see also \eqref{eq:Rdef}), set
\begin{align*}
\bar{R}_q^{(1)}&=\begin{cases}-\partial_t\chi_i\mathcal R(v_i-v_{i+1})+\chi_i(1-\chi_i)(v_i-v_{i+1})\mathring{\otimes}(v_i-v_{i+1})& t\in I_i,\\ 0&t\in J_i,\end{cases}\\
\bar{R}_q^{(2)}&=\sum_{i=\underline{n}-1}^{\overline{n}+1}\chi_iR_i=(\chi_{\underline{n}-1}+\chi_{\overline{n}+1})R_q,\\
\end{align*}
and 
$$
\bar{p}_q^{(2)}=\chi_i(1-\chi_i)\left(|v_i-v_{i+1}|^2-\fint_{\T}|v_i-v_{i+1}|^2\,dx\right).
$$
Finally, we define
$$
\bar{R}_q=\mathring{\bar{R}}_q^{(1)}+\mathring{\bar{R}}_q^{(2)}+\bar{\rho}_q\Id,\quad \bar{p}_q=\bar{p}_q^{(1)}+\bar{p}_q^{(2)},
$$
where
\begin{align}\label{eq:barrq}
\bar{\rho}_q=\rho_q+\frac13\Big(\fint_{\T}|v_q|^2-|\bar v_q|^2\Big).
\end{align}
By construction 
$$
\partial_t\bar{v}_q+\div(\bar{v}_q\otimes\bar{v}_q)+\nabla\bar{p}_q=-\div \bar{R}_q
$$
and  \eqref{eq:eq1} holds. 
Moreover 
$$
\mathring{\bar{R}}_q=0\quad\textrm{ for all }t\in \bigcup_{i=\underline{n}}^{\overline{n}}J_i.
$$

\bigskip

\subsection*{Step 3 - Stability estimates on classical solutions}
	
Let us consider for the moment $\underline{n}\leq i\leq \overline{n}$.
We recall from \cite[Proposition 3.1]{BDSV} that by the classical existence results on solutions of \eqref{eq:classic}, $(v_i,p_i)$ in \eqref{e:defbarvq} above is defined at least on an interval of length $\sim \|v_{\ell_{q,i}}\|_{1+\alpha}^{-1}$. By \eqref{eq:indata} and \eqref{eq:tauqk}
$$
\|v_{\ell_{q,i}}\|_{1+\alpha}\lesssim \delta_q^{1/2}\lambda_q^{1+\alpha}\leq \ell_q^{3\alpha}\tau_{q}^{-1},
$$
therefore indeed, provided $a\gg 1$ is sufficiently large, $v_i$ is defined on $I_{i-1}\cup J_i\cup I_i$ so that \eqref{e:defbarvq} is well defined. 

Next, we deduce from \eqref{eq:dtroq} that $|\partial_t\log\rho_q|\leq \delta_q^{1/2}\lambda_q=\tau_q^{-1}\ell_q^{4\alpha}$, so that, by assuming $a\gg 1$ is sufficiently large we may ensure that
\begin{equation}
\label{eq:rhononcost}
\rho(t_1)\leq 4 \rho(t_2)\quad\textrm{ for all }t_1,t_2\in I_{i-1}\cup J_i\cup I_i
\end{equation}
for any $i$. In particular $\rho_q\sim \rho_{q,i}$ in the interval $I_{i-1}\cup J_i\cup I_i$.
Then, reasoning as in \cite{BDSV}[Proposition 3.3], namely writing the transport equation along $v_{\ell{q,i}}$ for $v_i-v_{\ell{q,i}}$ and estimating the various terms on the left hand side (with the help of  \eqref{eq:indata} and \eqref{eq:Rln}), one reduces to a Gr\"onwall type inequality for the $C^{N+\alpha}$ norms of $v_i-v_{\ell{q,i}}$, namely
\begin{align*}
\|v_i-v_{\ell_{q,i}}\|_{N+\alpha}\lesssim \int_{t_i}^t\bigl(\tau_q^{-1} \|v_{\ell_{q,i}}-v_i\|_{N+\alpha}+\ell_{q,i}^{-N-1-\alpha}\rho_q^{1+\gamma}\bigr)\,ds.
\end{align*}
Using now the estimate \eqref{eq:rhononcost}, one obtains on $I_{i-1}\cup J_i\cup I_i$, as in \cite{BDSV}[Proposition 3.3],
\begin{equation}\label{eq:vivell}
\begin{split}
\|v_i-v_{\ell_{q,i}}\|_{N+\alpha}&\lesssim \tau_q\rho_{q,i}^{1+\gamma}\ell_{q,i}^{-N-1-\alpha}\\
&\lesssim \rho_{q,i}^{(1+\gamma)/2}\ell_{q,i}^{-N+\alpha}. 
\end{split}
\end{equation}
The case $N=0$, together with \eqref{eq:vlvq1} leads to \eqref{eq:eq2}, whereas the case $N=1$ leads to
$$
\|v_i-v_{\ell_{q,i}}\|_{1+\alpha}\lesssim \delta_q^{1/2}\lambda_q^{1+3\alpha/2}\ell_{q,i}^{\alpha}\leq \delta_q^{1/2}\lambda_q^{1+\alpha},
$$
so that, combining with \eqref{eq:vq1} and  with \eqref{eq:indata} we deduce that \eqref{eq:eq3} is verified.
More generally, following \cite{BDSV}[Proposition 4.3] we deduce from \eqref{eq:indata} and \eqref{eq:vivell} that
$$
\|\bar v_q\|_{1+N+\alpha}\lesssim \delta_q\lambda_q^{1+\alpha}\ell_{q,i}^{-N}
$$
for all $t$ in the region defined by \eqref{e:gluingregion}.
Thus \eqref{eq:estb1} is verified.

\bigskip

	\subsection*{Step 4 - Estimates on the new Reynolds stress}\hfill
	
Following \cite{BDSV} we define the vector potentials $z_i=(-\Delta)^{-1}\textrm{curl }v_i$, $z_{\ell_{q,i}}=(-\Delta)^{-1}\textrm{curl }z_{\ell_{q,i}}$ and obtain, as in \cite{BDSV}[Proposition 3.4] the analogous estimates to \eqref{eq:vivell}:
\begin{align*}
\|z_i-z_{\ell_{q,i}}\|_{N+\alpha}&\lesssim \tau_q\rho_{q,i}^{1+\gamma} \ell_{q,i}^{-N-\alpha}\,,\\
\|(\partial_t+v_{\ell_{q,i}}\cdot\nabla)(z_i-z_{\ell_{q,i}})\|_{N+\alpha}&\lesssim \rho_{q,i}^{1+\gamma} \ell_{q,i}^{-N-\alpha}\,
\end{align*}
valid in $I_{i-1}\cup J_i\cup I_i$ for any $\underline{n}\leq i\leq \overline{n}$. Proceeding as in the proof of \cite{BDSV}[Proposition 4.4] we deduce, using \eqref{eq:rhononcost}, that on $J_{\underline{n}}\cup\dots\cup J_{\overline{n}}$
\begin{equation}\label{e:Rqnew}
\begin{split}
\|\mathring{\bar R}_q\|_{N+\alpha}&\lesssim\tau_{q}^{-1}\|z_i-z_{i+1}\|_{N+\alpha}+\|v_i-v_{i+1}\|_{N+\alpha}\|v_i-v_{i+1}\|_{\alpha}\\
&\lesssim \rho_{q}^{1+\gamma}\ell_{q,i}^{-N-\alpha}\,,
\end{split}
\end{equation}
and similarly
\begin{align*}
\|(\partial_t+\bar{v}_q\cdot\nabla)\mathring{\bar R}_q\|_{N+\alpha}\lesssim \tau_q^{-1}\rho_{q}^{1+\gamma}\ell_{q,i}^{-N-\alpha}\,
\end{align*}
for all $t$ as in \eqref{e:gluingregion}. This shows that \eqref{eq:Rnalpha2} and \eqref{eq:last12} hold.	

Next, we estimate $\bar{\rho}_q$, recalling its definition in \eqref{eq:barrq}.
As in Proposition 4.5 of \cite{BDSV} one has that
\begin{align}
\Big|\frac{d}{dt}\int_{\T}|\bar v_q|^2-|v_{\ell_{q,i}}|^2\Big|&\lesssim\|v_{\ell_{q,i}}\|_1\|\mathring{\bar R}_{\ell_{q,i}}\|_0\lesssim \delta_q^{1/2}\lambda_q^{1+\alpha}\ell_{q,i}^{-\alpha}\rho_q^{1+\gamma}.\label{e:energydifferences}
\end{align}
Integrating \eqref{e:energydifferences} in $t\in I_{i-1}\cup J_i\cup I_i$ and using
\eqref{eq:vivlnorm} and \eqref{eq:ellqn} we deduce
\begin{equation*}
|\bar{\rho}_q-\rho_q|\lesssim  \rho_q^{1+\gamma}\ell_q^{3\alpha}\lambda_q^{\alpha}\lesssim  \rho_q^{1+\gamma}\ell_q^{2\alpha}.
\end{equation*}
This proves in particular that $\bar\rho_q\sim\rho_q$ and \eqref{eq:eq4-1}. 
Similarly, using the equation \eqref{eq:ER} for $(v_q,p_q,R_q)$ and $(v_{\ell_{q,i}},p_{\ell_{q,i}},\mathring{R}_{\ell_{q,i}})$, we also deduce
\begin{equation}\label{eq:eq4bis0}
\Big|\frac{d}{dt}\int_{\T}|v_{\ell_{q,i}}|^2-|v_q|^2\Big|\lesssim \delta_q^{1/2}\lambda_q^{1+\alpha}\ell_{q,i}^{-\alpha}\rho_q^{1+\gamma},
\end{equation}
hence
\[
\Big|\int_{\T}|v_{\ell_{q,i}}|^2-|v_q|^2\Big|\leq \rho_q^{1+\gamma}\ell_q^{2\alpha}
\]
and together with \eqref{e:energydifferences} \eqref{eq:eq4} is proved.
Combining \eqref{eq:eq4bis0} with \eqref{e:energydifferences}, \eqref{e:lambdaelli}  and \eqref{eq:rhoest} we obtain
\begin{align*}
|\partial_t\bar{\rho}_q-\partial_t\rho_q|&\lesssim  \rho_q^{1+\gamma}\delta_q^{1/2}\ell_{q,i}^{-\alpha}\lambda_q^{1+\alpha}\\
&\lesssim \rho_q\delta_q^{1/2}\lambda_q(\delta_{q+1}^{\gamma}\lambda_{q}^{\alpha(1+b)})\\
&\lesssim \rho_q\delta_q^{1/2}\lambda_q,
\end{align*}
where we have used \eqref{eq:dl2} and the assumption $\alpha b<\beta\gamma$ in the last line.
This shows \eqref{eq:last0}.

It remains to estimate $\|\mathring{\bar R}_q(t)\|_0$ on $[T_1,T_2]$ in order to verify \eqref{eq:eq4bis} for the Reynolds stress. Observe that we already obtained \eqref{e:Rqnew} on $J_{\underline{n}}\cup\dots\cup J_{\overline{n}}$ (recall \eqref{e:decomposition}). Moreover, on $J_{\underline{n}-1}\cup J_{\overline{n}+1}$ the subsolution remains unchanged, so there is nothing to prove. Finally, in the cut-off regions $I_{\underline{n}-1}$ and $I_{\overline{n}}$ we need to carry on the estimate \eqref{e:Rqnew} with $z_i=z_q$ and $v_i=v_q$, $z_{q}=(-\triangle)^{-1}\curl v_{q}$. In particular we need to estimate  $\|z_{\ell_{q,i}}-z_q\|_{\alpha}$. 
	One has that, by \eqref{e:mollify2}, Schauder estimates and \eqref{eq:vq1},
	\begin{align*}
	\|z_{\ell_{q,i}}-z_q\|_{\alpha}&\lesssim\|z_q\|_{2+\alpha}\ell_{q,i}^{2}\\
	&\lesssim\|\curl v_q\|_{\alpha}\ell_{q,i}^2\\
	&\lesssim\tau_q\rho_{q,i}^{1+\gamma}\ell_q^{-\alpha}.
	\end{align*}
	Therefore, \eqref{eq:eq4bis} follows.

\end{proof}

\begin{remark}\label{rem:moreint}
	Proposition \ref{prop:gluing} can easily be extended to a pairwise disjoint union of intervals $[T_1^{(i)},T_2^{(i)}]\subset [0,T]$ with
	$T_2^{(i)}-T_1^{(i)}\geq 4\tau_q$ and $T_2^{(i)}<T_1^{(i+1)}$.
\end{remark}

\section{Perturbation step}
\label{sect:pert}

\begin{proposition}\label{prop:pert}
	 Let $b,\beta,\alpha,\gamma$ and $(\delta_q,\lambda_q,\ell_q,\tau_q)$ be as in Section \ref{s:guide} with $\alpha/\gamma<2 \beta$. Let $[T_1,T_2]\subset [0,T]$ and let $(v,p,R)$ be a smooth strong subsolution on $[T_1,T_2]$. Further, let $S\in C^\infty(\T\times[T_1,T_2];S^{3\times3})$ be a smooth matrix field with
	 \begin{equation}\label{e:Ssigma0}
	 S(x,t)=\sigma(t)\Id+\mathring{S}(x,t),	
	 \end{equation}
	 where $\mathring{S}$ is traceless and $\sigma$ satisfies 
	 \begin{align}
	 0\leq\sigma(t)&\leq 4\delta_{q+1}\,,\label{eq:sigmadelta}\\
	 {|\partial_t\sigma|}& \lesssim\sigma \delta_q^{1/2}\lambda_q\,.\label{eq:dtsigmabar}
	 \end{align}
	 Moreover, assume that for any $N\geq 0$
	 \begin{align}
	 {\|\mathring{S}\|_{N+\alpha}}&\lesssim \sigma^{{\gamma}+1}\ell_{q}^{-N-\alpha}\,,\label{eq:sn}\\
	 \|v\|_{N+1+\alpha}&\lesssim \delta_q^{1/2}\lambda_q^{1+\alpha}\ell_{q}^{-N}\,,\label{eq:vn}\\
	 {\|(\partial_t+v\cdot\nabla)\mathring{S}\|_{N+\alpha}}&\lesssim \sigma^{{\gamma}+1}\ell_{q}^{-N-5\alpha}\delta_q^{1/2}\lambda_q\,.\label{eq:stimadttr}
	 \end{align}
	Finally, assume that
	 	\begin{equation}\label{e:suppSring}
	 \supp\,\mathring{S}\subset\T\times\bigcup_{i}I_i,	
	 \end{equation}
	 where $\{I_i\}_{i}$ are the intervals defined in \eqref{e:defIi}.
	 
	 Then, provided $a\gg 1$ is sufficiently large (depending on the implicit constants in \eqref{eq:dtsigmabar}-\eqref{eq:stimadttr}), there exist smooth $(\tilde v,\tilde p)\in C^\infty(\T\times[T_1,T_2];\R^3\times\R)$ and a smooth matrix field $\mathcal E\in C^\infty(\T\times[T_1,T_2];\mathcal S^{3\times 3})$, $\supp\,\mathcal E\subset\T\times\supp\,S$ such that, setting $\tilde R=R-S-\mathcal E$, the triple $(\tilde v,\tilde p,\tilde R)$ is a strong subsolution with
	\begin{equation}\label{eq:6.9}
	\int_{\T}|\tilde v|^2+\tr\tilde R\,dx=\int_{\T}|v|^2+\tr R\,dx\quad \forall\,t.
	\end{equation}
	Moreover, we have the estimates
	\begin{align}
	\|\tilde v-v\|_0&\leq \frac M2\delta_{q+1}^{1/2}\label{eq:v1v}\\
	\|\tilde v-v\|_{1+\alpha}&\leq\frac M2\delta_{q+1}^{1/2}\lambda_{q+1}^{1+\alpha}\label{eq:v1va}
	\end{align}
	where $M$ is a geometric constant, and the error $\mathcal E$ satisfies the estimates
	\begin{equation}\label{eq:eest}
	\|\mathcal E\|_0\leq \delta_{q+2}\lambda_{q+1}^{-3\alpha},
	\end{equation}
	\begin{equation}\label{eq:eestdt}
	|\partial_t\tr\mathcal E|\leq \delta_{q+2}\delta_{q+1}^{1/2}\lambda_{q+1}^{1-3\alpha}.
	\end{equation}
\end{proposition}

\begin{proof}
	The proof is a localization of the argument carried on in Section 5 of \cite{BDSV}.	The point is that the matrix field that has to be ``absorbed'' by the perturbation flow is not the whole $R$ as in \cite{BDSV} but $S$. 
	
\subsection*{Step 1 - Squiggling Stripes and the Stress Tensor $\tilde S_i$}
	
	Let $\{I_i\}_{i}$ be the intervals in \eqref{e:defIi} so that \eqref{e:suppSring} holds, and 
	set
	\[
	J_i=\Big(t_i-\frac13\tau_{q},t_i+\frac13\tau_{q}\Big).
	\]
	Following \cite{BDSV}[Lemma 5.3], we choose a family of smooth nonnegative $\eta_i=\eta_i(x,t)$ with the following properties:
	\begin{enumerate}
		\item [(i)] $\eta_i\in C^\infty(\T\times[T_1,T_2])$ with $0\leq\eta_i(x,t)\leq1$ for all $(x,t)$;\\
		\item [(ii)] $\supp\,\eta_i\cap\supp\,\eta_j=\emptyset$ for $i\neq j$;\\
		\item [(iii)] $\T\times I_i\subset\{(x,t):\,\eta_i(x,t)=1\}$;\\
		\item[(iv)] $\supp\,\eta_i\subset\T\times J_i\cup I_i\cup J_{i+1}=\T\times(t_i-\frac13\tau_{q},t_{i+1}+\frac13\tau_{q})\cap[0,T]$;\\
		\item [(v)] There exists a positive geometric constant $c_0>0$ such that, for any $t\in[0,T]$
		\[
		\underset{i}{\sum}\int_{\T}\eta_i^2(x,t)\,dx\geq c_0.		
		\]\\
		\item[(vi)] For $N,m\geq0$, $\|\partial_t^N\eta_i\|_m\leq C(N,m)\tau_{q}^{-N}$
		\end{enumerate}
	\bigskip
	
	Define 
	\[
	\sigma_i(x,t):=|\T|\frac{\eta_i^2(x,t)}{\underset{_j}{\sum}\int\eta_j^2(y,t)\,dy}\sigma(t),
	\]
	so that $\sum_i\int_{\T}\sigma_i\,dx=\int_{\T}\sigma\,dx$, and, using
	the inverse flow ${\Phi}_i$ starting at time $t_i$
	\begin{equation*}
	\left\{\begin{aligned}
	(\partial_t+v\cdot\nabla){\Phi}_i&=0\\
	 \Phi_i(x,t_i)&=x
	\end{aligned}\right.
	\end{equation*}
	set
	\begin{align*}
	S_i&=\sigma_i\Id+\eta_i^2\mathring S,\\
	\tilde S_i&=\frac{\nabla{\Phi}_iS_i(\nabla{\Phi}_i)^T}{\sigma_i}.
	\end{align*}
	One can check from the properties of $\eta_i$ and from \eqref{eq:sigmadelta} that 
	\begin{align*}
	\|\sigma_i\|_0&\leq\frac{4\delta_{q+1}}{c_0}\\
	\|\sigma_i\|_N&\lesssim \delta_{q+1}\,,
	\end{align*}
	and moreover, using \eqref{e:suppSring}, 
	\begin{equation}\label{eq:SiS}
	\frac13\tr\sum_i\int_{\T}S_i\,dx=\sigma=\frac13\tr S.
	\end{equation}
	
	We next claim that for all $(x,t)$
	\begin{equation}\label{e:tildeSivalues}
	\tilde S_i(x,t)\in B_{1/2}(\Id)\subset\mathcal S^{3\times3}_+,
	\end{equation}
	where $B_{1/2}(\Id)$ is the ball of radius $1/2$ centred at the identity $\Id$ in $\mathcal S^{3\times3}$.
	Indeed, by classical estimates on transport equations (see e.g.~\cite{BDSV}[Appendix B])
	\begin{equation}\label{e:DPhi-est}
	\|\nabla\Phi_i-\Id\|_0\lesssim\tau_{q}\delta_q^{1/2}\lambda_q^{1+\alpha}\leq \ell_{q}^{\alpha}\,
	\end{equation}
	for $t\in J_i\cup I_i\cup J_{i+1}$, since this is an interval of length $\sim\tau_q$.
	Using \eqref{eq:sigmadelta}, \eqref{eq:sn}  and \eqref{e:lambdaell} we also have, for any $N\geq 0$
	\begin{equation}\label{e:Ssigma-est}
	\bigg\|\frac{\eta_i^2\mathring S}{\sigma_i}\bigg\|_N\lesssim \bigg\|\frac{\mathring S}{\sigma}\bigg\|_N\lesssim \sigma^{\gamma}\ell_q^{-N-\alpha}\lesssim \delta_{q+1}^\gamma\lambda_{q+1}^{\alpha}\ell_q^{-N}=\lambda_{q+1}^{\alpha-2\beta\gamma}\ell_q^{-N}\,.
	\end{equation}
	Thus, using the decomposition 
	\begin{align*}
	\tilde S_i-\Id=\nabla{\Phi}_i\frac{\eta_i^2\mathring S}{\sigma_i}\nabla{\Phi}_i^T+\nabla{\Phi}_i\nabla{\Phi}_i^T-\Id
	\end{align*}
	we deduce
	\[
	|\tilde S_i-\Id|\lesssim \lambda_{q+1}^{\alpha-2\beta\gamma}+\ell_q^{ \alpha}\leq\frac{1}{2},
	\]
	provided $a\gg 1$ is sufficiently large, since we assumed $\alpha<2\beta\gamma$. This verifies \eqref{e:tildeSivalues}.

\subsection*{Step 2 - The perturbation $w$}

Now we can define the perturbation term as
	\begin{equation*}
	w_o:=\underset{i}{\sum}(\sigma_i)^{1/2}(\nabla{\Phi}_i)^{-1}W(\tilde S_i,\lambda_{q+1}{\Phi}_i)=\sum_iw_{oi},
	\end{equation*}
	where $W$ are the Mikado flows defined in Section \ref{subs:mikado}, see also Remark \ref{rem:Mconst}.
	Notice that the supports of the $w_{oi}$ are disjoint and, using the Fourier series representation of the Mikado flows \eqref{eq:mikFourier},
	\begin{equation}\label{e:woi}
	w_{oi}:=\sum_{k\neq0}(\nabla{\Phi}_i)^{-1}b_{i,k}e^{i\lambda_{q+1}k\cdot{\Phi}_i},
	\end{equation}
	where we write
	\[
	b_{i,k}(x,t):=(\sigma_i(x,t))^{1/2}a_k(\tilde S_i(x,t))A_k.
	\]	
	We define $w_c$ so that $w=w_o+w_c$ is divergence free:
	\begin{equation*}
	\begin{split}
	w_c:=\frac{-i}{\lambda_{q+1}}\sum_{i,k\neq0}\nabla((\sigma_i)^{1/2}a_k(\tilde S_i))\times\frac{\nabla\Phi_i^T(k\times A_k)}{|k|^2}e^{i\lambda_{q+1}k\cdot\Phi_i}=\sum_{i,k\neq0}c_{i,k}e^{i\lambda_{q+1}k\cdot\Phi_i}.
	\end{split}
	\end{equation*}
	Define then 
	\begin{align*}
	w&=w_o+w_c\\
	\tilde v&=v+w\\
	\tilde p&=p+|w|^2-\sum_i\sigma_i,\\
	\mathcal E(x,t)&=\mathring{ \mathcal E}^{(1)}(x,t)+\mathcal E^{(2)}(t), 
	\end{align*}
	where
	\begin{equation}
	\mathring{\mathcal E}^{(1)}:=\mathcal R\Bigl[\partial_t\tilde v+\div(\tilde v\otimes \tilde v)+\nabla \tilde p +\div(R-S)\Bigr],
	\end{equation}
with $\mathcal R$ being the operator defined in \eqref{eq:Rdef}, and
	\[
	\mathcal E^{(2)}(t):=\frac13\Big(\fint_{\T}|\tilde v|^2-|v|^2-\tr S\,dx\Big)\Id.
	\]
		Equations \eqref{eq:6.9} and \eqref{eq:ER} follow by construction.
	
\subsection*{Step 3 - Estimates on the perturbation}
	
	The estimates on $\tilde v$ and $\mathring{\mathcal E}^{(1)}$ follow similarly to the ones for $v_{q+1}$ and $\mR_{q+1}$ in Section 5 and 6 of \cite{BDSV}. As $\mathring{\mathcal E}^{(1)}$ is defined through the operator $\mathcal R$, in order to estimate its parts we use the stationary phase Proposition \ref{prop:stat}. In order to bound the terms involved we require analogous estimates to the ones in Section 5 of \cite{BDSV}. First of all, generalizing \eqref{e:DPhi-est}, for all $N\geq 0$ and $t\in J_i\cup I_i\cup J_{i+1}$
	\begin{align}
	\|(\nabla\Phi_i)^{-1}-\Id\|_N+\|\nabla\Phi_i-\Id\|_N&\lesssim \tau_q\delta_q^{1/2}\lambda_q^{1+\alpha}\ell_{q}^{-N}\nonumber\\
	&\lesssim \ell_q^{\alpha-N}\label{eq:phi1}\\
	\|(\partial_t+v\cdot\nabla)\nabla\Phi_i\|_N&\lesssim \|\nabla\Phi_i-\Id\|_0\|v\|_{N+1}+\|\nabla\Phi_i-\Id\|_N\|v\|_{1}\nonumber\\
	&\lesssim \ell_q^{\alpha-N}\delta_q^{1/2}\lambda_q^{1+\alpha}\nonumber\\
	&\lesssim\delta_q^{1/2}\lambda_q\ell_{q}^{-N},\label{eq:phi2}
	\end{align}
	where we used the identity
	$(\partial_t+v\cdot\nabla)\nabla\Phi_i=-(\nabla\Phi_i-\Id)Dv$, estimates \eqref{eq:vn} and \eqref{e:DPhi-est} and the fact the flow $\Phi_i$ is defined on a time interval of length $\tau_{q}$. Then, the following estimates follow precisely as in \cite{BDSV}[Propositions 5.7 and 5.9]:
	\begin{align}
	\|\tilde S_i\|_N&\lesssim\ell_{q}^{-N},\label{eq:pr1}\\
	\|b_{i,k}\|_N&\lesssim\delta_{q+1}^{1/2}|k|^{-6}\ell_{q}^{-N}\label{eq:pr2}\\
	\|c_{i,k}\|_{N}&\lesssim\delta_{q+1}^{1/2}\lambda_{q+1}^{-1}|k|^{-6}\ell_{q}^{-N-1}\label{eq:pr3}\\
	\|D_t\tilde S_i\|_N&\lesssim\tau_{q}^{-1}\ell_{q}^{-N}\label{eq:pr4}\\
	\|D_tc_{i,k}\|_N&\lesssim\delta_{q+1}^{1/2}\tau_{q}^{-1}\ell_{q}^{-N-1}\lambda_{q+1}^{-1}|k|^{-6}.\label{eq:pr5}
	\end{align}
	In obtaining \eqref{eq:pr1} we use \eqref{e:Ssigma-est} and the assumption that $\alpha<2\beta \gamma$.  
	In turn, from these estimates the estimates on $\tilde v$ in \eqref{eq:v1v}-\eqref{eq:v1va} follow precisely as in \cite{BDSV}[Corollary 5.8]. 

\subsection*{Step 4 - Estimates on the new Reynolds term $\mathring{\mathcal{E}}^{(1)}$}

The estimates for $\mathring{\mathcal{E}}^{(1)}$ are analogous to  those obtained for the new Reynolds stress in Section 6 of \cite{BDSV}. Therefore we obtain, using \eqref{e:lambdadelta},
\begin{equation}\label{e:E1-est}
\|\mathring{\mathcal{E}}^{(1)}\|_0\lesssim \frac{\delta_{q+1}^{1/2}\delta_q^{1/2}\lambda_q}{\lambda_{q+1}^{1-5\alpha}}\leq \delta_{q+2}\lambda_{q+1}^{-3\alpha}.
\end{equation}
	
\subsection*{Step 5 - Estimates on the new Reynolds term $\mathcal E^{(2)}$}
	
	Now we turn to $\mathcal E^{(2)}$, consider the decomposition
	\begin{equation}\label{e:E2split}
	\begin{split}
	|\mathcal E^{(2)}|&=\frac13\Big|\fint_{\T}|\tilde v|^2-|v|^2-\tr S\Big|\\
	&\leq\frac13\Big|\fint_{\T} |w_o|^2-\tr S\Big|+\frac13\Big|\fint_{\T} w_o\cdot w_c+w_c\cdot w_o+w_c\otimes w_c\Big|\\
	&+\frac13\Big|\fint_{\T}w\cdot v+v\cdot w\Big|
	\end{split}
	\end{equation}
	and proceed as in \cite{BDSV}[Proposition 6.2]. Concerning the first term in \eqref{e:E2split}, using \eqref{e:woi} and \eqref{eq:Ck} we have
	\begin{equation*}
	w_o\otimes w_o=\sum_iw_{oi}\otimes w_{oi}=\sum_iS_i+\sum_{i,k\neq0}\sigma_i\nabla\Phi_i^{-1}C_k(\tilde S_i)\nabla\Phi_i^{-T}e^{i\lambda_{q+1}k\cdot\Phi_i}\,.
	\end{equation*} 
	Using \eqref{eq:aphiint}, the properties of $C_k$ in \eqref{eq:Ck} and \eqref{e:lambdaellN} we obtain
	\begin{equation*}
	\Big|\int_{\T}\sum_{i,k\neq0}\sigma_i\nabla\Phi_i^{-1}C_k(\tilde S_i)\nabla\Phi_i^{-T}e^{i\lambda_{q+1}k\cdot\Phi_i}\Big|\lesssim \sum_{k\neq0}\frac{\delta_{q+1}\ell_{q}^{-\overline{N}}}{\lambda_{q+1}^{\overline{N}}|k|^{\overline{N}}}\lesssim \frac{\delta_{q+1}}{\lambda_{q+1}}\,.
	\end{equation*}
	Furthermore, as in \cite{BDSV}[Proposition 6.2], we also obtain
	\begin{equation*}
	\Big|\fint_{\T} w_o\otimes w_c+w_c\otimes w_o+w_c\otimes w_c\Big|
	+\Big|\fint_{\T}w\otimes v+v\otimes w\Big|
	\lesssim \frac{\delta_q^{1/2}\delta_{q+1}^{1/2}\lambda_q^{1+2\alpha}}{\lambda_{q+1}}\,.
	\end{equation*}
	Thus, combining with \eqref{eq:SiS} and \eqref{e:lambdadelta} we arrive at
	\begin{equation*}
	|\mathcal E^{(2)}|\lesssim  \frac{\delta_q^{1/2}\delta_{q+1}^{1/2}\lambda_q^{1+2\alpha}}{\lambda_{q+1}}\leq \frac{\delta_{q+2}}{\lambda_{q+1}^{6\alpha}}\,.
	\end{equation*}
	Combining with \eqref{e:E1-est} and taking $a\gg 1$ sufficiently large we thus verify 
	\eqref{eq:eest}. 
	
\subsection*{Step 6 - Estimates on $\partial_t\tr\mathcal{E}$}
Observe that $\mathring{\mathcal{E}}^{(1)}$ is traceless, whereas $\mathcal{E}^{(2)}$ is a function of $t$ only. 
	In order to  estimate the time derivative of $\mathcal E^{(2)}$, observe that, since $v$ is solenoidal, for every $F=F(x,t)$
	\[
	\frac{d}{dt}\int_{\T}F=\int_{\T}D_tF,
	\]
	where $D_t=\partial_t+v\cdot\nabla$.
	Therefore, using again the decomposition in \eqref{e:E2split}, we have
	\begin{align}\label{eq:stimadte}
	\Big|\frac{d}{dt}\int_{\T}\tilde v\otimes \tilde v-&v\otimes v-S\Big|\leq\Big|\int_{\T}D_t\Big(w_o\otimes w_c+w_c\otimes w_o+w_c\otimes w_c\Big)\Big|\nonumber\\
	&+\Big|\int_{\T}D_t\Big(w\otimes v+v\otimes w\Big)\Big|\\
	&+\Big|\int_{\T}D_t\Big(\sum_{i,k\neq0}\sigma_i\nabla \Phi_i^{-1} C_k(\tilde S_i)\nabla\Phi_i^{-T}e^{i\lambda_{q+1}k\cdot\Phi_i}\Big)\Big|\,.\nonumber
	\end{align}
Let us first estimate $\|D_tw_o\|_0$. 
\begin{align*}
D_tw_o&=\sum_{i,k\neq0}D_t\Big((\sigma_i)^{1/2}a_k(\tilde S_i)\Big)\nabla\Phi_i^{-1}A_ke^{i\lambda_{q+1}k\cdot\Phi_i}\\
&+\sum_{i,k\neq0}(\sigma_i)^{1/2}a_k(\tilde S_i)(\nabla v)^T\nabla\Phi_i^{-1}A_ke^{i\lambda_{q+1}k\cdot\Phi_i}\\
&=\sum_{i,k\neq0}d_{n,k} e^{i\lambda_{q+1}k\cdot\Phi_i}+\sum_{i,k\neq0}g_{i,k}e^{i\lambda_{q+1}k\cdot\Phi_i}.
\end{align*}
First notice that, by using \eqref{eq:vn}, \eqref{eq:sigmadelta}, \eqref{eq:phi1}, the estimates on the Fourier coefficients of the Mikado flows, and arguing as in \eqref{eq:phi2}, we obtain
\[
\|g_{i,k}\|_0\lesssim\frac{\delta_{q+1}^{1/2}\delta_q^{1/2}\lambda_q}{|k|^6}.
\]
As for the coefficients $d_{i,k}$, the estimate follows from \eqref{eq:pr4} and from
\[
\|D_t(\sigma_i^{1/2})\|_0\lesssim\tau_{q}^{-1}\delta_{q+1}^{1/2}.
\]
Therefore
\begin{equation*}
\|D_tw_o\|_0\lesssim{\delta_{q+1}^{1/2}\delta_q^{1/2}\lambda_q}\ell_{q}^{-4\alpha}\,.
\end{equation*}
Similarly we can deduce
\begin{equation*}
\|D_tw_c\|_0\lesssim{\delta_{q+1}^{1/2}\delta_q^{1/2}\lambda_q}\ell_{q}^{-4\alpha}\lambda_{q+1}^{-1}\,.
\end{equation*}
Combining with $\|w_o\|_0+\|w_c\|_0\lesssim\delta_{q+1}^{1/2}$ and using \eqref{e:lambdadelta}-\eqref{e:lambdaell}, we arrive at
\begin{align*}
\Big|\int_{\T}D_t\Big(w_o\otimes w_c+w_c\otimes w_o+w_c\otimes w_c\Big)\Big|&\lesssim \delta_{q+1}\delta_q^{1/2}\lambda_{q}\ell_q^{-4\alpha}\\
&\lesssim \delta_{q+2}\delta_{q+1}^{1/2}\lambda_{q+1}^{1-3\alpha}
\end{align*}
The estimate of the third term in \eqref{eq:stimadte} is entirely similar.
Finally, the estimate of the term involving $D_t(w\otimes v)$ follows by the estimates above on the terms given by $D_tw_o$ and the stationary phase Proposition \ref{prop:stat}. More precisely, we write
$$
D_t\Big(w\otimes v\Big)=\sum_{i,k\neq 0}h_{i,k}e^{i\lambda_{q+1}k\cdot\Phi_i},
$$
with
$$
\|h_{i,k}\|_{N}\lesssim \delta_{q+1}^{1/2}\delta_{q}^{1/2}\lambda_q\ell^{-4\alpha-N}_q,
$$
leading to
\begin{align*}
\Big|\int_{\T}D_t\Big(w\otimes v\Big)\Big|&\lesssim \frac{\delta_{q+1}^{1/2}\delta_{q}^{1/2}\lambda_q\ell^{-4\alpha-\overline{N}}_q}{\lambda_{q+1}^{\overline{N}}}\leq \frac{\delta_{q+1}^{1/2}\delta_q^{1/2}\lambda_{q}}{\lambda_{q+1}}\\
&\lesssim  \delta_{q+2}\delta_{q+1}^{1/2}\lambda_{q+1}^{1-6\alpha}
\end{align*} 
as before, using \eqref{e:lambdaellN} and the trivial estimate $1\leq \delta_{q+1}^{1/2}\lambda_{q+1}$. 
As a result, we obtain the estimate \eqref{eq:eestdt}.

\end{proof}

\section{From strict to adapted subsolutions}
\label{sect:stradapt}

The aim of this section is to prove Proposition \ref{prop:stradapt}. The proof is based on an iterative convex integration scheme similar in structure to that implemented in \cite{DanSz}. Here however, each stage contains  an additional localized gluing step and the estimates in the localized perturbation step are $1/3$-type estimates.

\begin{proof}[Proof of Proposition \ref{prop:stradapt}]\hfill


\subsection*{Step 1 - Setting the parameters of the scheme}
	Let $(v,p,R)$ be a smooth strict subsolution and
	let $0<\beta<1/3$, $\nu>0$ be as in the statement of the proposition.	
	Choose $b>1$ according to \eqref{eq:bbeta}, furthermore let $\bar\eps>0$ such that 
\begin{equation}\label{eq:beps}
	b(1+\bar\eps)<\frac{1-\beta}{2\beta}.
\end{equation}
Then, let $\tilde\delta,\tilde\gamma>0$ be the constants obtained in Corollary \ref{cor:strstr}, and choose $0<\alpha<1$ and $0<\gamma<\hat\gamma\leq\tilde\gamma$ so that the conditions in Section \ref{s:guide} and the inequalities \eqref{e:lambdadelta}-\eqref{e:lambdaellN} are satisfied,
	\begin{equation}\label{eq:hatnu}
		\nu>\frac{1-3{\beta}+\alpha}{2{\beta}}
	\end{equation} 
and furthermore
	\begin{equation}\label{eq:param}
    \frac{\alpha b}{\beta}<\hat{\gamma}<\frac{3\alpha}{2\beta},\quad 0<{\gamma}<\hat{\gamma}-\frac{\alpha}{2\beta}.
    \end{equation}
   Having fixed $b,\beta,\alpha,\gamma,\hat{\gamma}$ we may choose $\overline{N}\in\N$ so that \eqref{e:lambdaellN} is also valid. For $a\gg1$ sufficiently large (to be determined) we then define $(\lambda_q,\delta_q)$ as in \eqref{eq:dl2}.	Thus we are in the setting of Section \ref{s:guide}. 
	
\subsection*{Step 2 - From strict to strong subsolution}

We apply Corollary \ref{cor:strstr} to obtain from $(v,p,R)$ a strong subsolution $(v_0,p_0,R_0)$ with $\delta=\delta_1$ such that the properties \eqref{eq:str1}-\eqref{eq:str6} hold. We claim that, with such a choice of the parameters, $(v_0, p_0, R_0)$ satisfies
	\begin{align}
	\frac34\delta_1&\leq\rho_0 \leq\frac54\delta_1\label{eq:v01}\\
	\|\mathring{ R}_0(t)\|_0&\leq\rho_0^{1+\hat\gamma}\label{eq:R3}\\
		\|v_0\|_{1+\alpha}&\leq \delta_0^{1\slash2}\lambda_0^{1+\alpha}\label{eq:barv3bis}\\
	|\partial_t\rho_0|&\leq\delta_1\delta_0^{1/2}\lambda_0\,.\label{eq:dtrho3}
	\end{align}	
	Indeed, \eqref{eq:v01} and \eqref{eq:R3} follow directly from \eqref{eq:str1}-\eqref{eq:str2} since $\delta=\delta_1$. In order to verify \eqref{eq:barv3bis}-\eqref{eq:dtrho3} we need to choose $\tilde\eps>0$ in \eqref{eq:str3}-\eqref{eq:str6} so that
	$$
	\delta_1^{-(1+\tilde\eps)}\leq \delta_0^{1/2}\lambda_0.
	$$
	According to the definition of $(\lambda_q,\delta_q)$ this is valid by our choice of $\tilde\eps$ in \eqref{eq:beps} above. In turn, the constants in \eqref{eq:str3}-\eqref{eq:str6} can be absorbed by a sufficiently large $a\gg 1$. 
	
	\subsection*{Step 3 - Inductive construction of $(v_q,p_q,R_q)$}
	
	 Starting from $(v_0,p_0,R_0)$, we show how to construct inductively a sequence $\{(v_q,p_q,R_q)\}_{q\in\N}$ of smooth strong subsolutions with
	\[
	R_q(x,t)=\rho_q(t)\Id+\mR_q(x,t)
	\]
	which satisfy the following properties:
	\begin{enumerate}
		\item [$(a_q)$] For all $t\in[0,T]$
		\begin{equation*}
		\int_{\T}|v_q|^2+\tr R_q=\int_{\T}|v_0|^2+\tr R_0;
		\end{equation*}
		\item [$(b_q)$] For all $t\in[0,T]$
		\begin{equation*}
		\|\mR_q(t)\|_0\leq\rho_q^{1+\gamma};
		\end{equation*}
		\item [$(c_q)$] If $2^{-j}T<t\leq2^{-j+1}T$ for some $j=1,\dots,q$, then
		\begin{equation*}
		\frac38\delta_{j+1}\leq\rho_q\leq4\delta_{j};
		\end{equation*}
		\item [$(d_q)$] For all $t\leq 2^{-q}T$ 
		\begin{equation*}
		\|\mR_q(t)\|_0\leq\rho_q^{1+\hat{\gamma}}, \quad \tfrac34\delta_{q+1}\leq\rho_q\leq\tfrac54\delta_{q+1};
		\end{equation*}
		\item [$(e_q)$] If $2^{-j}T<t\leq2^{-j+1}T$ for some $j=1,\dots,q$, then
		\begin{align*}
		\|v_q\|_{1+\alpha}&\leq M\delta_j^{1/2}\lambda_j^{1+\alpha}\,,\\
		|\partial_t\rho_q|&\leq \delta_{j+1}\delta_j^{1/2}\lambda_j\,,
		\end{align*}
		whereas if $t\leq 2^{-q}T$, 
		\begin{align*}
		\|v_q\|_{1+\alpha}&\leq M\delta_q^{1/2}\lambda_q^{1+\alpha}\,,\\
		|\partial_t\rho_q|&\leq  \delta_{q+1}\delta_q^{1/2}\lambda_q\,.
		\end{align*} 
		\item [$(f_q)$] For all $t\in[0,T]$  and $q\geq 1$
		\begin{equation*}
		\|v_q-v_{q-1}\|_0\leq\frac{M}{2}\delta_q^{1/2}.
		\end{equation*}
	\end{enumerate}
	Thanks to our choice of parameters in Step 1 above, $(v_0,p_0,R_0)$ satisfies \eqref{eq:v01}-\eqref{eq:dtrho3} and therefore our inductive assumptions $(a_0)-(f_0)$. 
	
	\bigskip
	
	Suppose then $(v_q,p_q,R_q)$ is a smooth strong subsolution satisfying $(a_q)-(f_q)$. The construction of $(v_{q+1},p_{q+1},R_{q+1})$ is done in two steps: first a localized gluing step performed using Proposition \ref{prop:gluing} to get from $(v_{q},p_{q},R_{q})$ a smooth strong subsolution $(\bar v_q,\bar p_q,\bar R_q)$,  then a localized perturbation step done using Proposition \ref{prop:pert}  to get  $(v_{q+1},p_{q+1},R_{q+1})$ from $(\bar v_q,\bar p_q,\bar R_q)$.

	We apply Proposition \ref{prop:gluing} with 
	$$
	[T_1,T_2]=[0,2^{-q}T].
	$$
	Then $T_2-T_1\geq 4\tau_q$, provided $a\gg 1$ sufficiently large. Moreover, by $(d_q)-(e_q)$ and \eqref{eq:param}, $(v_q,p_q,R_q)$ fulfils the requirements of Proposition \ref{prop:gluing} on $[T_1,T_2]$ with parameters $\alpha,\hat{\gamma}>0$.
	
	Then, by Proposition \ref{prop:gluing} we obtain a smooth strong subsolution $(\bar v_q,\bar p_q,\bar R_q)$ on $[0,T]$ such that $(\bar v_q,\bar p_q,\bar R_q)$ is equal to $(v_q,p_q,R_q)$ on $[2^{-q}T,T]$ and on $[0,2^{-q}T]$ satisfies
	\begin{equation}\label{eq:vbarest}
	\begin{split}	
		\|\bar v_q-v_q\|_\alpha&\lesssim\bar\rho_q^{(1+\hat\gamma)/2}\ell_{q}^{\alpha/3}\,,\\
	\|\bar v_q\|_{1+\alpha}&\lesssim\delta_q^{1/2}\lambda_q{1+\alpha}\,,\\
	\|\mathring{\bar R}_q\|_{0}& \leq\bar\rho_q^{1+{\hat\gamma}}\ell_{q}^{-\alpha}\,,\\
	\tfrac58\delta_{q+1}&\leq\bar \rho_q\leq\tfrac32\delta_{q+1}\,,\\
	|\partial_t\bar\rho_q|&\lesssim  \delta_{q+1}\delta_q^{1/2}\lambda_q\,.
	\end{split}
	\end{equation} 
	Moreover, on $[0,  t_{\overline{n}}]$  one has
	\begin{equation}\label{eq:Rbarest}
	\begin{split}
	\|\bar v_q\|_{N+1+\alpha}&\lesssim\delta_q^{1/2}\lambda_q^{1+\alpha}\ell_q^{-N}\,,\\
	\|{\mathring{\bar  R}_q}\|_{N+\alpha}&\lesssim \bar\rho_q^{1+{\hat\gamma}}\ell_{q}^{-N-\alpha}\,,\\
	\|(\partial_t+\bar v_q\cdot\nabla)\mathring{\bar R}_q\|_{N+\alpha}&\lesssim\bar\rho_q^{1+{\hat\gamma}}\ell_{q}^{-N-\alpha}\delta_q^{1/2}\lambda_q.
	\end{split}
	\end{equation} 
	and
		\begin{equation}\label{eq:Rbarsupp}
		\mathring{\bar R}_q\equiv 0\quad\textrm{ for }t\in \bigcup_{i=0}^{\overline{n}}J_i.
		\end{equation}
Recalling Definition \ref{def:intervals} and \eqref{eq:tntn} observe that
\begin{equation}\label{eq:suppS}
[0,\tfrac342^{-q}T]\subset [0,t_{\overline{n}}],
\end{equation} 
provided $a\gg 1$ is chosen sufficiently large. Then, 
fix a cut-off function $\psi_q\in C^\infty_c([0,\tfrac342^{-q}T);[0,1])$	such that
	\begin{equation}
	\psi_q(t)=\begin{cases}
	1 & t\leq 2^{-(q+1)}T,\\
	0 & t>\tfrac34 2^{-q}T,\end{cases}
	\end{equation}
	and such that $|\psi'_{q}(t)|\lesssim 2^q$. By choosing $a\gg 1$ sufficiently large, we may assume that 
	\begin{equation}\label{eq:psiqder}
	|\psi'_{q}(t)|\leq\frac12\delta_q^{1/2}\lambda_q
	\end{equation}
	for all $q$. Then, set 
\[
 S=\psi_{q}^2(\bar R_q-\delta_{q+2}\Id).
\]
Using \eqref{eq:psiqder}, \eqref{eq:param} and \eqref{eq:vbarest}-\eqref{eq:suppS} we see that $S$ and $(\bar{v}_q,\bar{p}_q,\bar{R}_q)$ satisfy the assumptions of Proposition \ref{prop:pert} on the interval $[0,t_{\overline{n}}]$ with parameters $\alpha,\hat{\gamma}>0$. 

Proposition \ref{prop:pert} gives then a new subsolution $(v_{q+1},p_{q+1}, \bar R_q-S-\tilde{\mathcal E}_{q+1})$ with 
\begin{align*}
\|v_{q+1}-\bar{v}_q\|_0+\lambda_{q+1}^{-1-\alpha}\|v_{q+1}-\bar{v}_q\|_{1+\alpha}&\leq \frac{M}{2}\delta_{q+1}^{1/2}\,,\\
\int_{\T}|v_{q+1}|^2-\tr S-\tr\tilde{\mathcal E}_{q+1}=\int_{\T}|\bar v_q|^2&\quad\textrm{ for all }t\in[0,T], 
\end{align*}
and such that the estimates \eqref{eq:eest}-\eqref{eq:eestdt} hold for 
$\tilde{\mathcal E}_{q+1}$. 
Let
\[
R_{q+1}:=\bar R_q-S-\tilde{\mathcal E}_{q+1}.
\]

\smallskip

We claim that $(v_{q+1}, p_{q+1}, R_{q+1})$ is a smooth strong subsolution satisfying $(a_{q+1})-(f_{q+1})$. Notice that $(a_{q+1})$ is satisfied by construction. Since $(v_{q+1},p_{q+1},R_{q+1})=(v_q,p_q,R_q)$ for $t\geq 2^{-q}T$, we may restrict to $t\leq 2^{-q}T$ in the following, so that in particular \eqref{eq:vbarest} holds. 

Let us now prove $(b_{q+1})$. On the one hand
\begin{align*}
\|\mR_{q+1}\|_0&=\|(1-\psi_{q}^2)\mathring{\bar R}_q-\mathring{\mathcal E}_{q+1}\|_0\\
&\leq(1-\psi_{q}^2)\bar\rho_q^{1+\hat{\gamma}}\ell_{q}^{-\alpha}+\delta_{q+2}\lambda_{q+1}^{-3\alpha}\mathbbm{1}_{\{\psi_{q}>0\}},
\end{align*}
on the other hand
\begin{align*}
\rho_{q+1}&=(1-\psi_{q}^2)\bar\rho_q+\psi_{q}^2\delta_{q+2}+\tfrac13\tr\mathcal E_{q+1}\mathbbm{1}_{\{\psi_{q}>0\}}\\
&\geq(1-\psi_{q}^2)\bar\rho_q+\psi_{q}^2\delta_{q+2}-\delta_{q+2}\lambda_{q+1}^{-3\alpha}\mathbbm{1}_{\{\psi_{q}>0\}}.
\end{align*}
The question  is then whether there exists a suitable ${\gamma}$ such that
\begin{equation}\label{eq:gammaineq}
\begin{split}	
(1-\psi_{q}^2)&\bar\rho_q^{1+\hat{\gamma}}\ell_{q}^{-\alpha}+\delta_{q+2}\lambda_{q+1}^{-3\alpha}\mathbbm{1}_{\{\psi_{q}>0\}}\\
&\leq[(1-\psi_{q}^2)\bar\rho_q+\psi_{q}^2\delta_{q+2}-\delta_{q+2}\lambda_{q+1}^{-3\alpha}\mathbbm{1}_{\{\psi_{q}>0\}}]^{1+{\gamma}}.
\end{split}
\end{equation}
To this end set
\begin{align*}
F(s)&=(1-s)\bar{\rho}_q^{1+\hat\gamma}\ell_q^{-\alpha}+\delta_{q+2}\lambda_{q+1}^{-3\alpha}\,,\\
G(s)&=(1-s)\bar{\rho}_q+s\delta_{q+2}-\delta_{q+2}\lambda_{q+1}^{-3\alpha}\,,
\end{align*}
and observe that \eqref{eq:gammaineq} is equivalent to $F(\psi_{q}^2)\leq G^{1+\gamma}(\psi_{q}^2)$ if $\psi_{q}>0$, and follows from this inequality also in case $\psi_{q}=0$. In particular, \eqref{eq:gammaineq} follows from
\begin{enumerate}
\item[(i)] $F(0)\leq G^{1+\gamma}(0)$;
\item[(ii)]	$F'(s)\leq (1+\gamma)G^{\gamma}(s)G'(s)$.
\end{enumerate}
We note next that, since $2\beta\hat{\gamma}<3\alpha$, 
$$
\delta_{q+2}\lambda_{q+1}^{-3\alpha}\lesssim \bar{\rho}_q^{1+\hat\gamma},
$$
so that we have the estimates
$$
F(0)\lesssim \bar{\rho}_q^{1+\hat\gamma}\ell_q^{-\alpha},\quad G(0)\gtrsim \bar\rho_q
$$
and also clearly $G(s)\leq \bar\rho_q$. 
Then is it easy to check that (i) amounts to
$\bar{\rho}_q^{1+\hat{\gamma}}\ell_q^{-\alpha}\lesssim \bar{\rho}_q^{1+\gamma}$, and hence (using \eqref{e:lambdaell}) follows from
$$
\hat{\gamma}-\frac{\alpha}{2\beta}>\gamma
$$
whereas (ii) follows from $\gamma< \hat{\gamma}$, provided $a\gg 1$ is sufficiently large to absorb geometric constants. Thus, our choice of $\gamma$ in \eqref{eq:param} guarantees that \eqref{eq:gammaineq}, and hence $(b_{q+1})$ is satisfied.  

Consider now $(c_{q+1})$, where we only need to consider the case $j=q+1$, i.e.~the estimate on $[2^{-q-1}T, 2^{-q}T]$. Arguing as above, we see that
 \begin{equation*}
\delta_{q+2}(1-\lambda_{q+1}^{-3\alpha})\leq \rho_{q+1}(t)\leq \bar{\rho}_q(t)+\delta_{q+2}\lambda_{q+1}^{-3\alpha}\leq \tfrac32\delta_{q+1}+\delta_{q+2}\lambda_{q+1}^{-3\alpha}.
\end{equation*}
Therefore $(c_{q+1})$ holds, provided $a\gg 1$ is sufficiently large.

Similarly, concerning $(d_{q+1})$ note that for $t\leq 2^{-(q+1)}T$ we have $\psi_q(t)=1$, so that
$$
\delta_{q+2}(1-\lambda_{q+1}^{-3\alpha})\leq \rho_{q+1}\leq \delta_{q+2}(1+\lambda_{q+1}^{-3\alpha}).
$$
Moreover, as above, for $t\leq 2^{-(q+1)}T$ 
$$
\|\mathring{R}_{q+1}\|_0\leq \delta_{q+2}\lambda_{q+1}^{-3\alpha}\leq (\tfrac34\delta_{q+1})^{1+\hat\gamma},
$$
since $2\beta\bar\gamma<3\alpha$ and by choosing $a\gg 1$ sufficiently large.
Therefore $(d_{q+1})$ holds.

Concerning $(e_{q+1})$ it suffices again to restrict to $t\leq 2^{-q}T$, i.e.~the case $j=q+1$. From \eqref{eq:vbarest} and \eqref{eq:v1va} we deduce
\begin{align*}
\|v_{q+1}\|_{1+\alpha}&\leq \|v_{q+1}-\bar{v}_q\|_{1+\alpha}+ \|\bar{v}_q\|_{1+\alpha}\\
&\leq \frac{M}{2}\delta_{q+1}^{1/2}\lambda_{q+1}^{1+\alpha}+C\delta_q^{1/2}\lambda_q^{1+\alpha}\\
&\leq M \delta_{q+1}^{1/2}\lambda_{q+1}^{1+\alpha},
\end{align*}
where $C$ is the implicit constant in \eqref{eq:vbarest} which can be absorbed by choosing $a\gg 1$ sufficiently large. In a similar manner the estimate on $|\partial_t\rho_{q+1}|$ follows from  and \eqref{eq:eestdt}. 
Finally, $(f_{q+1})$ follows from \eqref{eq:vbarest} and \eqref{eq:v1v}.

\subsection*{Step 4. Convergence to an adapted subsolution}

We have then obtained a sequence $(v_q,p_q,R_q)$ satisfying $(a_q)-(f_q)$. 

From $(f_q)$, it follows that $(v_q,p_q)$ is a Cauchy sequence in $C^0$. Indeed, for $\{v_q\}$ this is clear. Regarding $\{p_q\}$ we may use \eqref{eq:ER} to write
$$
\Delta (p_{q+1}-p_q)=-\div\div\left(\mathring{R}_{q+1}-\mathring{R}_q+(v_{q+1}-v_q)\otimes v_q+v_{q+1}\otimes (v_{q+1}-v_q)\right),
$$ 
and use Schauder estimates. Similarly, also $\{R_q\}$ converges in $C^0$. Indeed, from the definition and using \eqref{eq:eq4} and  we have
\begin{align*}
\|R_{q+1}-R_q\|_0&=\|\bar{R}_q-R_q-S-\tilde{\mathcal{E}}_{q+1}\|_0\\
&\leq \|\bar R_q-R_q\|_0+\|S\|_0+\|\tilde{\mathcal{E}}_{q+1}\|_0\\
&\lesssim \delta_{q+1}.
\end{align*} 
Since for all $t>0$ there exists $q(t)\in\N$ such that 
\[
(v_q,p_q,R_q)(\cdot,t)=(v_{q(t)},p_{q(t)},R_{q(t)})(\cdot,t) \quad\forall\,q\geq q(t),
\]
then $(v_q,p_q,R_q)$ converges uniformly to $(\hat v, \hat p, \hat R)$ where $(\hat v, \hat p, \hat R)$ is a strong subsolution with 
\[
\|\hat R\|_0\leq \hat{\rho}^{1+\gamma}
\]
and, using \eqref{eq:str5} and $(a_q)$ 
\[
\int_{\T}|\hat v|^2+\tr\hat R=\int_{\T} |v|^2+ \tr R\quad\textrm{ for all }t\in[0,T].
\]
Furthermore, using \eqref{eq:str4} and $(f_q)$
\begin{align*}
\|\hat v-v\|_{H^{-1}}&\leq \|v_0-v\|_{H^{-1}}+\|v_0-\hat v\|_{H^{-1}}\\
&\lesssim \delta_1+\sum_{q=0}^{\infty}\|v_{q+1}- v_q\|_0\\
&\lesssim \delta_1^{1/2},
\end{align*}
and similarly
\begin{align*}
\|\hat v\otimes\hat v +\hat R-v\otimes v-R\|_{H^{-1}}&\leq \| v_0\otimes v_0 +R_0-v\otimes v-R\|_{H^{-1}}+\\
&\qquad+\|\hat v\otimes\hat v +\hat R-v_0\otimes v_0-R_0\|_{H^{-1}}\\
&\lesssim \delta_1+\|\hat v\otimes\hat v-v_0\otimes v_0\|_{0}+\|\hat R\|_0+\|R\|_0\\
&\lesssim \delta_1.
\end{align*}

Concerning the initial datum, from $(e_q)$ and $(f_q)$ we obtain by interpolation that $\hat v(\cdot,0)\in C^\beta$, and from $(d_q)$ we obtain that $\hat R(\cdot,0)=0$.

Finally, we verify conditions \eqref{eq:adaptv}-\eqref{eq:adaptdtr} for being an adapted subsolution. Let $t>0$. Then there exists $q\in\N$ such that $t\in[2^{-q}T, 2^{-q+1}T]$. By $(c_q)$ and $(e_q)$ 
\begin{align*}
\frac38\delta_{q+1}&\leq\hat{\rho}\leq4\delta_{q},\\
\|\hat v\|_{1+\alpha}&\leq M\delta_q^{1/2}\lambda_q^{1+\alpha}.	
\end{align*}
Therefore $\hat\rho^{-1}\geq\frac14\delta_q^{-1}$ and hence, using \eqref{eq:dl2} and \eqref{eq:hatnu}, we deduce 
\begin{align*}
\|\hat v\|_{1+\alpha}&\leq \hat{\rho}^{-(1+\nu)}
\end{align*}
for $a\gg 1$ sufficiently large.
Similarly, using $(e_q)$ and \eqref{eq:hatnu} we deduce
\begin{equation*}
|\partial_t\hat{\rho}|\leq \delta_{q+1}\delta_q^{1/2}\lambda_q\leq \delta_q^{1-\frac{1-\beta}{2\beta}}\leq \hat{\rho}^{-\nu}\,.
\end{equation*}
This completes the proof of Proposition \ref{prop:stradapt}.
\end{proof}

\section{From adapted subsolutions to solutions}
\label{sect:adaptsol}

The aim of this section is to prove Proposition \ref{prop:adaptsol}.  We will start now from an adapted subsolution and we will build through a convex integration scheme a sequence of strong subsolutions converging to a solution of the incompressible Euler equations. Here as in Proposition \ref{prop:stradapt} the convex integration scheme will need localized gluing and perturbation arguments, namely Propositions \ref{prop:gluing} and \ref{prop:pert}. However, the choice of the cut-off functions will be, as in \cite{DanSz}, dictated by the shape of the trace part of the Reynolds stress and not fixed a priori as in Proposition \ref{prop:stradapt}.

\begin{proof}[Proof of Proposition \ref{prop:adaptsol}: ]\hfill
	
\subsection*{Step 1 - Setting of parameters in the scheme.}

Let $(\hat v,\hat p,\hat R)$ be a $C^{\beta}$-adapted subsolution on $[0,T]$ satisfying \eqref{eq:RtrR} for some $\gamma>0$ and \eqref{eq:adaptv}-\eqref{eq:adaptdtr} for some $\alpha,\nu>0$ as in Definition \ref{d:adapted}, with 
\[
\frac{1-\beta}{2\beta}<1+\nu<\frac{1-\hat\beta}{2\hat\beta}.
\]
Fix $b>1$ so that
\begin{equation}\label{eq:bcond}
 b^2(1+\nu)<\frac{1-\hat\beta}{2\hat\beta}, \quad 2\hat\beta(b^2-1)<1.
\end{equation}
Observe that both \eqref{eq:RtrR} and \eqref{eq:adaptv}-\eqref{eq:adaptdtr} remain valid for any $\gamma'\leq\gamma$ and $\alpha'\leq\alpha$ (c.f.~Remark \ref{rem:gamma}). 
Then, we may assume that $\alpha,\gamma>0$ are sufficiently small, so that 
$(\bar v,\bar p,\bar R)$ satisfies \eqref{eq:RtrR} and \eqref{eq:adaptv}-\eqref{eq:adaptdtr} with these parameters, and furthermore choose $\tilde\gamma>0$ so that
\begin{equation}
\label{eq:gamma1ad}
\frac{\alpha b}{\beta}<{\gamma}<\frac{3\alpha}{2\beta},\quad \frac{\alpha b}{\beta}<\tilde \gamma<\gamma-\frac{\alpha}{2\beta }.
\end{equation}
Finally, having fixed $b,\beta,\hat\beta,\alpha,\gamma,\tilde\gamma$ we may choose $\overline{N}\in\N$ so that \eqref{e:lambdaellN} holds. 
For $a\gg1$ sufficiently large (to be determined) we then define $(\lambda_q,\delta_q)$ as in \eqref{eq:dl2}. Thus, we are in the setting of Section \ref{s:guide}.

\subsection*{Step 2 - The first approximation}

Let $(\hat v,\hat p,\hat R)$ be as in the statement of the proposition and fix $\eta=\delta_1$ (observe that $\delta_1$ depends on our choice of $a\gg 1$ which will be chosen sufficiently large in the subsequent proof; thus, if necessary we choose $\eta$ smaller than given in the statement of the proposition - this is certainly no loss of generality). We apply Corollary \ref{cor:adapt} to obtain another $C^\beta$-adapted subsolution $(v_0,p_0,R_0)$ with parameters $\gamma$, $\nu$ such that
$$
\rho_0\leq \eta/4\quad \textrm{ and }\quad v_0=\hat v\textrm{ for }t=0.
$$	
Observe that strictly speaking in applying Corollary \ref{cor:adapt} we would obtain a parameter $\gamma'<\gamma$. However, without loss of generality we may assume that the parameter is $\gamma$, since in Step 1 above we already chose $\gamma$ ``sufficiently small''. 
Furthermore
\begin{equation*}
\begin{split}
\int_{\T}|v_0|^2+\tr R_0=\int_{\T}|\hat v|^2+\tr \hat R&\quad\textrm{ for all }t,\\
\|v_0-\hat v\|_{H^{-1}}&\leq \eta/2,\\
\|v_0\otimes v_0+R_0-\hat v\otimes \hat v-\hat R\|_{H^{-1}}&\leq \eta/2.
\end{split}
\end{equation*}

We claim that then the following holds: for any $q\in\N$ and any $t\in[0,T]$ such that $\rho_0(t)\geq \delta_{q+2}$, we have
\begin{equation}\label{eq:hatvest}
\begin{split}
\|v_0\|_{1+\alpha}&\leq \delta_q^{1/2}\lambda_{q}^{1+\alpha},\\
|\partial_t\rho_0|&\leq \rho_0 \delta_q^{1/2}\lambda_q\,.
\end{split}
\end{equation}
Indeed, assuming $\rho(t)\geq \delta_{q+2}$ for some $q$, we obtain using \eqref{eq:dl2} and \eqref{eq:bcond}
\begin{align*}
\rho_0^{-(1+\nu)}(t)&\lesssim \lambda_q^{2\beta b^2(1+\nu)}\leq\delta_{q}^{1/2}\lambda_q\,,\\
\rho_0^{-\nu}(t)&\lesssim \lambda_q^{2\beta b^2\nu}\leq \delta_{q+2}\delta_q^{1/2}\lambda_q\,,
\end{align*}
provided $a\gg 1$ is sufficiently large to absorb constants.

\subsection*{Step 3 - Inductive construction of $(v_q,p_q,R_q)$.}

Starting with $(v_0,p_0,R_0)$ we 
construct inductively a sequence of $(v_q,p_q,R_q)$ of smooth strong subsolutions $q=1,2\dots$ with
	\[
	R_q(x,t)=\rho_q(t)\Id+\mR_q(x,t)
	\]
satisfying the following properties:

\begin{enumerate}
	\item [$(a_q)$] For all $t\in[0,T]$
	\begin{equation}
	\int_{\T}|v_q|^2+\tr R_q=\int_{\T}|v_0|^2+\tr R_0;
	\end{equation}
	\item [$(b_q)$] For all $t\in[0,T]$
	\begin{equation}
	\rho_q\leq \tfrac54\delta_{q+1};
	\end{equation}
	\item [$(c_q)$] For all $t\in[0,T]$
	\begin{equation}
	\|\mR_q\|_0\leq\begin{cases}\rho_q^{1+\tilde\gamma}& \textrm{ if }\rho_q\geq\tfrac32\delta_{q+2},\\
\rho_q^{1+\gamma}& \textrm{ if }\rho_q\leq\tfrac32\delta_{q+2};\end{cases}
	\end{equation}
	\item [$(d_q)$] If $\rho_q\geq\delta_{j+2}$ for some $j\geq q$, then
	\begin{align}
	\|v_q\|_{1+\alpha}&\leq M\delta_j^{1/2}\lambda_j^{1+\alpha},
	\label{eq:adsol1}\\
	|\partial_t\rho_q|&\leq\rho_q\delta_j^{1/2}\lambda_j;\label{eq:adsol3}
	\end{align}
	\item [$(e_q)$] For all $t\in[0,T]$ and $q\geq 1$
	\begin{equation}
	\|v_q-v_{q-1}\|_0\lesssim \delta_{q}^{1/2}.	\end{equation}
\end{enumerate}  
Thanks to our choice of parameters in Step 1 above, $(v_0,p_0,R_0)$ satisfies \eqref{eq:hatvest} and therefore our inductive assumptions $(a_0)-(f_0)$. 

Suppose now $(v_q,p_q,R_q)$ satisfies $(a_q)$-$(e_q)$ above. Let
\begin{align*}
J_q:=\Big\{t\in[0,T]:\,\rho_q(t)>\tfrac32\delta_{q+2}\Big\},\quad 
K_q:=\{t\in[0,T]:\,\rho_q(t)\geq 2\delta_{q+2}\}.
\end{align*}
Being (relatively) open in $[0,T]$, $J_q$ is a disjoint, possibly countable, union of (relatively) open intervals $(T_1^{(i)},T_2^{(i)})$. Let
$$
\mathcal{I}_q:=\left\{i:\,(T_1^{(i)},T_2^{(i)})\cap K_q\neq\emptyset\right\}
$$
and let $t_0\in (T_1^{(i)},T_2^{(i)})\cap K_q$ for some $i\in\mathcal{I}_q$. Since $K_q$ is compact, we may assume that the open interval $(T_1^{(i)},t_0)$ is contained in $J_q\setminus K_q$.
Using $(d_q)$ we then have
\begin{align*}
\tfrac32\delta_{q+2}=\rho_q(T_1^{(i)})&\geq\rho_q(t_0)-|T_1^{(i)}-t_0|\sup_{J_q}|\partial_t\rho_q|\\
&\geq2\delta_{q+2}-2\delta_{q+2}\delta_q^{1/2}\lambda_q|T_1^{(i)}-t_0|,
\end{align*}
hence 
\begin{equation}\label{eq:tt0}
|T_1^{(i)}-t_0|\geq \tfrac{1}{4}(\delta_q^{1/2}\lambda_q)^{-1}> 4\tau_q,
\end{equation}
provided $a\gg 1$ is chosen sufficiently large. Similar estimate holds with $T_2^{(i)}$. Therefore $T_2^{(i)}-T_1^{(i)}> 4\tau_q$ for any $i\in \mathcal{I}_q$, so that $\mathcal{I}_q$ is a finite index set.

Next, we apply Proposition \ref{prop:gluing} (in the form of Remark \ref{rem:moreint}) to $(v_q,p_q,R_q)$ on the intervals 
\[
\bigcup_{i\in\mathcal{I}_q}J_{q,i}.
\]
Since $\rho_q>\tfrac32\delta_{q+2}$ on $J_q$, from $(a_q)-(e_q)$ we see that the assumptions of Proposition \ref{prop:gluing} on $(v_q,p_q,R_q)$ hold with parameter $\tilde\gamma$. Then we obtain $(\bar v_q,\bar p_q,\bar R_q)$ such that
\begin{align*}
\|\bar v_q(t)-v_q(t)\|_\alpha&\lesssim\bar\rho_q^{(1+\tilde\gamma)/2}\ell_{q}^{\alpha/3},\\
\|\bar v_q\|_{1+\alpha}&\lesssim\delta_q^{1/2}\lambda_q^{1+\alpha},\\
\|\mathring{\bar R}_q\|_{0}& \leq\bar\rho_q^{1+{\tilde\gamma}}\ell_{q}^{-\alpha},\\
\tfrac78\rho_q\leq &\bar{\rho}_q\leq \tfrac98\rho_q,\\
|\partial_t\bar{\rho}_q|&\lesssim \bar\rho_q\delta_q^{1/2}\lambda_q\,.
\end{align*}
Moreover, recalling \eqref{eq:tntn}, for any $i\in\mathcal{I}_q$ we have the following additional estimates valid for $t\in [T_1^{(i)}+2\tau_q,T_2^{(i)}-2\tau_q]$:
	\begin{equation}\label{eq:barv-est1}
	\begin{split}
\|\bar v_q\|_{N+1+\alpha}&\lesssim\delta_q^{1/2}\lambda_q^{1+\alpha}\ell_{q}^{-N}\,,\\
\Big\|{\mathring{\bar  R}_q}\Big\|_{N+\alpha}&\lesssim \bar\rho_q^{1+{\tilde\gamma}}\ell_{q}^{-N-\alpha}\,,\\
\|(\partial_t+\bar v_q\cdot\nabla)\mathring{\bar R}_q\|_{N+\alpha}&\lesssim\bar\rho_q^{1+{\tilde\gamma}}\ell_{q}^{-N-\alpha}\delta_q^{1/2}\lambda_q\,,	
	\end{split}
	\end{equation}
and
\begin{equation}\label{eq:suppS1}
\supp\,\mathring{S}\subset\T\times\bigcup_{i}I_i,	
\end{equation}
where $\{I_i\}_{i}$ are the intervals defined in \eqref{e:defIi}.
Let us choose a cut-off function $\psi_{q}\in C^\infty_c(J_q;[0,1])$ such that
 \begin{align}
 \supp\,\psi_{q}&\subset \bigcup_{i\in\mathcal{I}_q}(T_1^{(i)}+2\tau_q,T_2^{(i)}-2\tau_q)\\
 K_q&\subset\{\psi_{q}=1\}\\
 |\psi_{q}'|&\lesssim \frac{1}{\delta_q^{1/2}\lambda_q}.\label{eq:psiqder1}
 \end{align}
 Such choice is made possible by \eqref{eq:tt0}. We want then to apply Proposition \ref{prop:pert} to $(\bar v_q,\bar p_q,\bar R_q)$ with  
 \[
  S=\psi_{q}^2(\bar R_q-\delta_{q+2}\Id)
 \]
 hence $\sigma=\psi_{q}^2(\bar \rho_q-\delta_{q+2})$. Using \eqref{eq:psiqder1}, \eqref{eq:gamma1ad} and \eqref{eq:barv-est1}-\eqref{eq:suppS1} we see that $S$ and $(\bar{v}_q,\bar{p}_q,\bar{R}_q)$ satisfy the assumptions of Proposition \ref{prop:pert} on the interval $[T_1^{(i)}+2\tau_q,T_2^{(i)}-2\tau_q]$ with parameters $\alpha,\tilde{\gamma}>0$. 
  Proposition \ref{prop:pert} gives then a new subsolution $(v_{q+1},p_{q+1}, \bar R_q-S-\tilde{\mathcal E}_{q+1})$ with 
\begin{align*}
\|v_{q+1}-\bar{v}_q\|_0+\lambda_{q+1}^{-1-\alpha}\|v_{q+1}-\bar{v}_q\|_{1+\alpha}&\leq \frac{M}{2}\delta_{q+1}^{1/2}\,,\\
\int_{\T}|v_{q+1}|^2-\tr S-\tr\tilde{\mathcal E}_{q+1}=\int_{\T}|\bar v_q|^2&\quad\textrm{ for all }t\in[0,T]. 
\end{align*}
and such that the estimates \eqref{eq:eest}-\eqref{eq:eestdt} hold for 
$\tilde{\mathcal E}_{q+1}$. 
Let
\[
R_{q+1}:=\bar R_q-S-\tilde{\mathcal E}_{q+1}.
\]

\smallskip 

We claim that $(v_{q+1}, p_{q+1}, R_{q+1})$ is a smooth strong subsolution satisfying $(a_{q+1})-(f_{q+1})$. Notice that $(a_{q+1})$ is satisfied by construction.
By definition of $S$, one has then 
 \begin{align*}
 \rho_{q+1}&=\bar{\rho}_q(1-\psi_{q}^2)+\psi_{q}^2\delta_{q+2}-\tfrac13\tr\tilde{\mathcal E}_{q+1}\,,\\
 \mathring{R}_{q+1}&=\mathring{\bar R}_q(1-\psi_{q}^2)-\mathring{\tilde{\mathcal E}}_{q+1}\,.
 \end{align*}
 Then $(b_{q+1})$ follows directly from \eqref{eq:eest} and the fact that $K_q\subset\{\psi_{q}=1\}$.
 
 Next, observe that if $\rho_{q+1}\leq\tfrac32\delta_{q+3}$, then $t\notin J_q$, hence $\rho_{q+1}=\rho_q$, $\mR_{q+1}=\mR_q$. Therefore in verifying conditions 
 $(c_{q+1})-(d_{q+1})$ it suffices to restrict to the case when $\rho_{q+1}\geq \tfrac32\delta_{q+3}$ and $j=q+1$, respectively. 
 
If $t\in J_q$ then the argument for showing $(c_{q+1})$ is precisely as the proof of $(b_{q+1})$ in Step 3 of Proposition \ref{prop:stradapt} above. Also, the estimates in $(d_{q+1})$ for $j=q+1$ follow from \eqref{eq:v1va} and \eqref{eq:eestdt}. Finally, $(e_{q+1})$ follows precisely as $(f_{q+1})$ in the proof of Proposition \ref{prop:stradapt} above.

Thus, the inductive step is proved.  

\smallskip

Finally, the convergence of $\{v_q\}$ to a solution of the Euler equations as in the statement of Proposition \ref{prop:adaptsol} follows easily from the sequence of estimates in $(a)_q-(f)_q$, analogously to Step 4 of Proposition \ref{prop:stradapt} above.

	\end{proof}

\end{document}